\documentclass{amsart}

\usepackage{amssymb}
\usepackage{amsthm}
\usepackage{amsmath}
\usepackage{nicefrac}

\usepackage{enumitem}

\usepackage{tikz}
\usetikzlibrary{arrows,calc}

\usepackage[foot]{amsaddr}

\usepackage{hyperref}

\usepackage{mathrsfs}

\theoremstyle{plain}
\newtheorem{proposition}{Proposition}[section]
\newtheorem{theorem}[proposition]{Theorem}
\newtheorem{lemma}[proposition]{Lemma}
\newtheorem{corollary}[proposition]{Corollary}
\theoremstyle{definition}

\newtheorem{definition}[proposition]{Definition}
\newtheorem{observation}[proposition]{Observation}
\theoremstyle{remark}
\newtheorem{remark}[proposition]{Remark}

\DeclareMathOperator{\Aff}{Aff}
\DeclareMathOperator{\Aut}{Aut}

\DeclareMathOperator{\supp}{supp}

\DeclareMathOperator{\GL}{GL}

\DeclareMathOperator{\Euc}{Euc}

\DeclareMathOperator{\dist}{dist}

\DeclareMathOperator{\Cc}{\mathcal{C}}

\DeclareMathOperator{\Hc}{\mathcal{H}}

\DeclareMathOperator{\Lc}{\mathcal{L}}

\DeclareMathOperator{\Tc}{\mathcal{T}}

\DeclareMathOperator{\Bb}{\mathbb{B}}
\DeclareMathOperator{\Cb}{\mathbb{C}}
\DeclareMathOperator{\Db}{\mathbb{D}}

\DeclareMathOperator{\Nb}{\mathbb{N}}

\DeclareMathOperator{\Rb}{\mathbb{R}}

\DeclareMathOperator{\Bf}{\mathsf{B}}
\DeclareMathOperator{\Gf}{\mathsf{G}}

\newcommand{\abs}[1]{\left|#1\right|}

\newcommand{\norm}[1]{\left\|#1\right\|}
\newcommand{\wt}[1]{\widetilde{#1}}

\newcommand{\ip}[1]{\left\langle #1\right\rangle}

\begin{document}

\title[Domains with bounded intrinsic geometry]{Compactness of the $\bar{\partial}$-Neumann problem on domains with bounded intrinsic geometry}
\author{Andrew Zimmer}\address{Department of Mathematics, University of Wisconsin, Madison, WI, USA}
\email{amzimmer2@wisc.edu}
\date{\today}
\subjclass[2010]{}

\begin{abstract} By considering intrinsic geometric conditions, we introduce a new class of domains in complex Euclidean space. This class is invariant under  biholomorphism and includes strongly pseudoconvex domains, finite type domains in dimension two, convex domains, $\Cb$-convex domains, and homogeneous domains. For this class of domains, we show that compactness of the $\bar{\partial}$-Neumann operator on $(0,q)$-forms is equivalent to the boundary not containing any $q$-dimensional analytic varieties (assuming only that the boundary is a topological submanifold). We also prove, for this class of domains, that the Bergman metric is equivalent to the Kobayashi metric and that the pluricomplex Green function satisfies certain local estimates in terms of the Bergman metric. \end{abstract}

\maketitle

\section{Introduction}

This paper is motivated by work of Fu-Straube~\cite{FS1998} who showed that for convex domains, compactness of the $\bar{\partial}$-Neumann operator on $(0,q)$-forms is equivalent to the boundary containing no $q$-dimensional analytic varieties. Our goal is to define a class of domains which contains the bounded convex domains, is invariant under biholomorphisms, and where we can prove the same result about compactness of the $\bar{\partial}$-Neumann operator. We define such a class as follows. 

\begin{definition}\label{defn:BBG} A domain $\Omega \subset \Cb^d$ has \emph{bounded intrinsic geometry} if there exists a complete K\"ahler metric $g$ on $\Omega$ such that 
\begin{enumerate}[label=(b.\arabic*)]
\item\label{item:bd_sec} the metric $g$ has bounded sectional curvature and positive injectivity radius,  
\item\label{item:SBG} there exists a $\Cc^2$ function $\lambda : \Omega \rightarrow \Rb$ such that  the Levi form of $\lambda$ is uniformly bi-Lipschitz to  $g$ and $\norm{\partial \lambda}_{g}$ is bounded on $\Omega$. 
\end{enumerate}
\end{definition}

The above conditions on the K\"ahler metric $g$ are intrinsic and hence having bounded intrinsic geometry is invariant under biholomorphism. Property~\ref{item:SBG} is motivated by Gromov's definition of K\"ahler hyperbolicity~\cite{Gromov1991}, McNeal's results on plurisubharmonic functions with self bounded complex gradient~\cite{M2002}, and vanishing results for $L^2$ cohomology~\cite{DF1983,D1994,D1997,M2002b}.

Many domains have bounded intrinsic geometry, including
\begin{enumerate}
\item strongly pseudoconvex domains, 
\item finite type domains in $\Cb^2$, 
\item convex domains or more generally $\Cb$-convex domains which are Kobayashi hyperbolic (with no boundary regularity assumptions), 
\item simply connected domains which have a complete K\"ahler metric with pinched negative sectional curvature, 
\item homogeneous domains, and
\item the Teichm\"uller space  of hyperbolic surfaces of genus $g$ with $n$ punctures.
\end{enumerate} 
Further, by definition, any domain biholomorphic to one of the domains listed above also has bounded intrinsic geometry. In Section~\ref{sec:examples}, we will describe these examples in more detail and give references. 

A domain $\Omega \subset \Cb^d$ has several standard invariant K\"ahler (pseudo-)metrics. For instance, if $\Bf_\Omega$ denotes the Bergman kernel on a domain $\Omega \subset \Cb^d$, then the Bergman (pseudo-)metric is defined by 
\begin{align*}
g_\Omega( v,w ) = \sum_{1 \leq i,j \leq d} \frac{\partial^2 \log \Bf_\Omega(z,z)}{\partial z_i \partial \bar{z}_j}v_i \bar{w}_j. 
\end{align*}
The K\"ahler metric in Definition~\ref{defn:BBG} does not apriori have to be one of the standard invariant K\"ahler metrics, but we will prove that a domain has bounded intrinsic geometry if and only if the Bergman metric satisfies the conditions in Definition~\ref{defn:BBG}.

\begin{theorem}\label{thm:bergman_reduction}(see Theorem~\ref{thm:bergman_metric_bounded}) If $\Omega \subset \Cb^d$ is a domain, then the following are equivalent: 
\begin{enumerate}
\item $\Omega$ has bounded intrinsic geometry,
\item the Bergman metric $g_\Omega$ satisfies Definition~\ref{defn:BBG}.
\end{enumerate}
Moreover, in this case
\begin{align*}
\sup_{z \in \Omega} \norm{\nabla^m R}_{g_\Omega} < \infty
\end{align*}
for all $m \geq 0$ where $R$ is the curvature tensor of $g_\Omega$. 
\end{theorem}

The ``moreover'' part says that the Bergman metric on a domain with bounded intrinsic geometry has bounded geometry in the standard Riemannian sense. 

As a corollary to Theorem~\ref{thm:bergman_reduction} and a result of Bremermann~\cite{B1955} we see that every domain with bounded intrinsic geometry is pseudoconvex.

\begin{corollary} A domain with bounded intrinsic geometry is pseudoconvex. \end{corollary}

\begin{remark} We will actually establish that a domain with bounded intrinsic geometry is pseudoconvex before proving Theorem~\ref{thm:bergman_reduction}. In particular, in Theorem~\ref{thm:comp_to_kob}, we will show that the Kobayashi distance on such a domain is Cauchy complete and hence by a result of Wu~\cite{W1967}, domains with bounded intrinsic geometry must be pseudoconvex.
\end{remark} 

In the context of Definition~\ref{defn:BBG}, we should mention the following well-known properties of the Bergman metric. If bounded pseudoconvex domain has Lipschitz boundary, then the Bergman metric is complete~\cite{C1999,H1999}. Also, the holomorphic sectional curvature of the Bergman metric is always bounded from above by 2~\cite{B1948,K1959} and the sectional curvatures are determined by the holomorphic sectional curvatures, so having bounded sectional curvature is equivalent to having holomorphic sectional  curvature bounded below.

\subsection{Analytic properties} Given a pseudoconvex $\Omega \subset \Cb^d$ and $1 \leq q \leq d$, let $L^2_{(0,q)}(\Omega)$ denotes the space of $(0,q)$-forms with square integrable coefficients and let $\bar{\partial}^*$ denote the $L^2$ adjoint of $\bar{\partial}$. The $\bar{\partial}$-Neumann operator $N_q : L^2_{(0,q)}(\Omega) \rightarrow L^2_{(0,q)}(\Omega)$ is the bounded inverse to the unbounded self-adjoint  surjective operator $\square := \bar{\partial}\bar{\partial}^*+ \bar{\partial}^*\bar{\partial}$ on $L^2_{(0,q)}(\Omega)$. These operators have been extensively studied and we refer the reader to~\cite{FK1972,Krantz1992,BS1999,CS2001, S2010} for details.

For domains with bounded intrinsic geometry, we will characterize the compactness of the $\bar{\partial}$-Neumann operator in terms of the growth rate of the Bergman metric. In particular, given a $d$-by-$d$ complex matrix $A$ let
\begin{align*}
\sigma_1(A) \geq \sigma_2(A) \geq \dots \geq \sigma_d(A)
\end{align*}
denotes the singular values of $A$. We will then prove the following. 

\begin{theorem}\label{thm:compactness_intro}(see Theorem~\ref{thm:compactness}) Suppose $\Omega \subset \Cb^d$ is a bounded domain with bounded intrinsic geometry. Then the following are equivalent: 
\begin{enumerate}
\item $N_q$ is compact.
\item If $g_{\Omega,z}$ is identified with the $d$-by-$d$ matrix $\left[ g_{\Omega,z}(\frac{\partial}{\partial z_i}, \frac{\partial}{\partial \bar{z}_j})\right]$, then 
\begin{align*}
\lim_{z\rightarrow \partial\Omega} \sigma_{d-q+1}( g_{\Omega,z})  =\infty.
\end{align*}
\end{enumerate}
If, in addition, $\partial \Omega$ is $\Cc^0$, then the above conditions are equivalent to: 
\begin{enumerate}
\setcounter{enumi}{2}
\item $\partial \Omega$ does not contain any $q$-dimensional analytic varieties.
\end{enumerate}
\end{theorem}

\begin{remark}  To be precise: \begin{enumerate} \item We say that $\partial\Omega$ is $\Cc^{r}$ (respectively $\Cc^{r,\alpha}$) if for every point $x \in \partial\Omega$ there exists a neighborhood $U$ of $x$ and there exists a linear change of coordinates which makes $U \cap \partial \Omega$ the graph of a $\Cc^{r}$ (respectively $\Cc^{r,\alpha}$) function. 
\item We say that $\partial \Omega$ contains a $q$-dimensional analytic variety if there exists a holomorphic map $\varphi: \Db^q \rightarrow \partial\Omega$ where $\varphi^\prime(0)$ has rank $q$. 
\end{enumerate}
\end{remark}

Convex domains always have $\Cc^{0,1}$ boundary and, as mentioned above, have bounded intrinsic geometry. In this special case, $(2) \Leftrightarrow (3)$ follows from estimates of Frankel~\cite{F1991} while $(1) \Leftrightarrow (3)$ was established by Fu-Straube~\cite{FS1998}. We also note that by earlier results of Henkin-Iordan~\cite{HI1997} and Sibony~\cite{Sib1987}, the $\bar{\partial}$-Neumann operator is compact for every $1 \leq q \leq d$ on a $B$-regular domain, a class of domains which includes bounded convex domains whose boundaries do not contain any $1$-dimensional analytic varieties. 

As an example, Theorem~\ref{thm:compactness_intro} implies the following extension of Fu-Straube's result. 

\begin{corollary} Suppose $\Omega \subset \Cb^d$ is a bounded domain with $\Cc^0$ boundary. If $\Omega$ is biholomorphic to a $\Cb$-convex domain (e.g. a convex domain), then the following are equivalent: 
\begin{enumerate}
\item $N_q$ is compact,
\item $\partial \Omega$ contains no $q$-dimensional analytic varieties.
\end{enumerate} 
\end{corollary}

\subsection{Geometric properties} We will also establish some geometric properties of domains with bounded intrinsic geometry. Our main result in this direction is that the Bergman metric and Kobayashi metric are equivalent. 

\begin{theorem}\label{thm:comp_to_kob_intro}(see Theorem~~\ref{thm:comp_to_kob} and~\ref{thm:bergman_metric_bounded}) If $\Omega \subset \Cb^d$ is domain with bounded  intrinsic geometry and $k_\Omega$ is the Kobayashi metric on $\Omega$, then there exists $C > 1$ such that 
\begin{align*}
\frac{1}{C} k_\Omega(z;v) \leq \sqrt{g_{\Omega,z}(v,v)} \leq C k_\Omega(z;v)
\end{align*}
for all $z \in \Omega$ and $v \in \Cb^d$. 
\end{theorem}

\begin{remark} This equivalence of metrics is a key part of the proof that $(2) \Rightarrow (3)$ in Theorem~\ref{thm:compactness_intro}.\end{remark}

We will also establish the following uniform local estimate for the pluricomplex Green function in terms of the Bergman distance. 

\begin{theorem}(see Theorem~\ref{thm:green_fcn_bds}) Suppose $\Omega \subset \Cb^d$ is a domain with bounded intrinsic geometry, $\dist_\Omega$ is the Bergman distance on $\Omega$, and $\Gf_\Omega$ is the pluricomplex Green function on $\Omega$. There exist $C, \tau> 0$ such that: \begin{align*}
\log \dist_\Omega(z,w) - C \leq  \Gf_\Omega(z,w) \leq \log \dist_\Omega(z,w) + C
\end{align*}
for all $z,w \in \Omega$ with $\dist_\Omega(z,w) \leq \tau$. 
\end{theorem}

\subsection{Potentials for the Bergman metric} Theorem~\ref{thm:bergman_reduction} says there is no loss in generality in considering only the Bergman metric in Definition~\ref{defn:BBG} and so it seems natural to wonder if one can simply consider the standard potential for the Bergman metric in Property~\ref{item:SBG}. Unfortunately, as the next proposition shows, this is not the case. 

\begin{proposition}(see Proposition~\ref{prop:not_inv}) There exists a bounded domain $\Omega \subset \Cb^2$ biholomorphic to $\Db \times \Db$ where
\begin{align*}
\norm{\partial \log \Bf_\Omega(z,z)}_{g_\Omega}
\end{align*}
is unbounded. 
 \end{proposition}
 
 It is easy to verify that $\Db \times \Db$, and hence also $\Omega$ in the Proposition, has bounded intrinsic geometry. So the above Proposition justifies the complicated formulation of Property~\ref{item:SBG}.

\subsection{Motivation for Definition~\ref{defn:BBG}}\label{sec:motivations} The definition of bounded intrinsic geometry is partially motivated by results of Catlin~\cite{Catlin1989} for finite type domains in $\Cb^2$ and McNeal~\cite{M1994,M1992} for finite type convex domains.  A central component of their work is the construction of certain embedded polydisks and associated plurisubharmonic functions. 

In particular, given such a domain $\Omega$ they show, essentially\footnote{In the $d=2$ case, the affine maps are defined in terms of holomorphic coordinates on $\Cb^d$ which depend on $\zeta$, see~\cite[Section 1]{Catlin1989}.}, that for every $\zeta \in \Omega$ there exists an affine embeddings $\Phi_\zeta : \Db^d \rightarrow \Omega$ of the form
\begin{align*}
\Phi_\zeta(z) = \zeta + U_\zeta \begin{pmatrix} \tau_1(\zeta) & & \\ & \ddots & \\ & & \tau_d(\zeta) \end{pmatrix} z
\end{align*}
(where $U_\zeta$ is a unitary matrix) and there exists a companion plurisubharmonic function $\phi_\zeta : \Omega \rightarrow [-1,1]$ with
\begin{align*}
 \sum_{1 \leq i,j \leq d} \frac{\partial^2 \phi_\zeta}{\partial z_i \partial \bar{z}_j}\xi_i\bar{\xi_j} \gtrsim \sum_{j=1}^d \frac{1}{\tau_j(\zeta)^2} \abs{\xi_j}^2 
\end{align*}
on $\Phi_\zeta(\Db^d)$. Then the plurisubharmonic functions $\phi_\zeta$ are used  to reduce global problems on $\Omega$ to local problems on $\Phi_\zeta(\Db^d)$.

In Section~\ref{sec:charts} we will show that domains with bounded intrinsic geometry have similar embeddings and plurisubharmonic functions. 

\begin{theorem}\label{thm:charts_intro}(see Theorem~\ref{thm:charts}) Suppose $\Omega \subset \Cb^d$ is a bounded pseudoconvex domain and $g$ is a complete K\"ahler metric on $\Omega$. 
\begin{enumerate}
\item If $g$ has Property~\ref{item:bd_sec}, then there exists $ A_1>1$ such that: For every $\zeta \in \Omega$ there exists a holomorphic embedding $\Phi_\zeta : \Bb \rightarrow \Omega$ with  $\Phi_\zeta(0) = \zeta$ and
\begin{align*}
\frac{1}{A_1} g_{\Euc} \leq \Phi_\zeta^* g \leq A_1 g_{\Euc}.
\end{align*}
\item  If $g$ has Property~\ref{item:SBG}, $\dist_g$ is the distance induced by $g$, and $r > 0$, then there exists $ A_2=A_2(r)>0$ such that: For every $\zeta \in \Omega$ there exists a plurisubharmonic function $\phi_\zeta : \Omega \rightarrow \Rb$ with
\begin{enumerate}
\item $\Lc(\phi_\zeta) \geq g$ on $B_g(\zeta; r):=\{z \in \Omega : \dist_g(z,\zeta) < r\}$,
\item $-A_2 \leq \phi_\zeta \leq  0$ on $\Omega$. 
\end{enumerate}
\end{enumerate}
\end{theorem} 

The existence of the embeddings in part (1) will follow from classical work of Shi~\cite{Shi1989} concerning regularization of Riemannian metrics and recent work of Wu-Yau~\cite{WY2020} concerning K\"ahler manifolds with bounded geometry. The plurisubharmonic functions in part (2) will be constructed in a direct way from the potential in Property~\ref{item:SBG}.  As in the work of Catlin and McNeal, Theorem~\ref{thm:charts_intro} will allow us to reduce global problems to local ones.

 \subsection*{Acknowledgements} I would like to thank Nessim Sibony and Sai Kee Yeung for a number of helpful comments. I would also like to thank Xieping Wang for pointing out a mistake in an earlier version of this paper. This material is based upon work supported by the National Science Foundation under grants DMS-1942302 and DMS-1904099.

\section{Examples}\label{sec:examples}

In this section we give precise references for the examples of domains with bounded intrinsic geometry listed in the introduction. 

\subsection{Finite type domains in dimension two}
Suppose $\Omega \subset \Cb^2$ is a smoothly bounded pseudoconvex domain of finite type. Since $\Omega$ has smooth boundary, the Bergman metric is complete~\cite{C1999,H1999}. McNeal~\cite{McNeal1989} proved that the Bergman metric has bounded holomorphic sectional curvature and hence bounded sectional curvature (the holomorphic sectional curvatures determine the sectional curvatures). Catlin established precise estimates for the Bergman metric near the boundary~\cite{Catlin1989}. Using these estimates, the fact that the sectional curvature is bounded, and Proposition 2.1 in~\cite{LSY2004b} one can show that the injectivity radius of the Bergman metric is positive. Thus Property~\ref{item:bd_sec} holds. Donnelly~\cite{D1997}, using Catlin's estimates, proved that 
\begin{align*}
\norm{\partial \log \Bf_\Omega(z,z)}_{g_\Omega}
\end{align*}
is uniformly bounded and thus Property~\ref{item:SBG} holds. 

\subsection{Domains with negatively curved K\"ahler metrics} Results of Greene-Wu~\cite{GW1979}  imply the following. 

\begin{theorem} Suppose $\Omega \subset \Cb^d$ is a simply connected domain and there exists a complete K\"ahler metric $g$ on $\Omega$ with 
\begin{align*}
-a^2 \leq \sec(g) \leq -b^2 <0
\end{align*}
for some constants $a,b > 0$. Then $\Omega$ has bounded intrinsic geometry. 
\end{theorem}

\begin{proof} Since $\Omega$ is simply connected and $g$ is negatively curved, the injectivity radius of $g$ is infinite by the Cartan-Hadamard theorem. So $g$ satisfies Property~\ref{item:bd_sec}. 

We will establish Property~\ref{item:SBG} using comparison theorems from~\cite{GW1979}. Let $\dist_g$ be the distance induced by $g$. Fix $o \in\Omega$. Since $g$ is negatively curved and $\Omega$ is simply connected, $\dist_g(\cdot, o)$ is $\Cc^\infty$ on $\Omega \setminus \{o\}$ (this also follows from the Cartan-Hadamard theorem).  Then let $\rho$ be a smooth real valued function on $\Omega$ such that $\rho(z) = \dist_g(z,o)$ when $\dist_g(z,o) \geq 1$. By the Hessian comparison theorem there exists $C > 1$ such that the Levi form of $\rho$ satisfies 
\begin{align*}
\frac{1}{C} g \leq \Lc(\rho) \leq C g
\end{align*}
on $\{ z \in \Omega: \dist_g(z,o) > 1\}$, see~\cite[Lemma 1.13 and Theorem A]{GW1979}. Further, the function $\phi(z) = \left(\tanh \frac{b\dist_g(z,o)}{2} \right)^2$ is smooth and strictly plurisubharmonic on $\Omega$, see~\cite[Example 6.15]{GW1979}. So for $M$ large and by possibly increasing $C > 1$ we have 
\begin{align*}
\frac{1}{C} g \leq \Lc(\rho+M\phi) \leq C g
\end{align*}
on $\Omega$. Finally, there exists $A > 0$ such that 
\begin{align*}
\norm{\partial( \rho+M\phi)}_{g} \leq A\norm{\partial_z \dist_g(z,o)}_{g} \leq A\norm{d_z \dist_g(z,o)}_{g} \leq A
\end{align*}
on $\{ z \in \Omega: \dist_g(z,o) > 1\}$. Hence $\lambda := \rho+M\phi$ satisfies Property~\ref{item:SBG}.
\end{proof}

\subsection{Holomorphic homogeneous regular domains} The other domains listed in the introduction are all holomorphic homogeneous regular domains and for such domains Property~\ref{item:bd_sec} always holds. 

\begin{definition}\cite{LSY2004a} A domain $\Omega \subset \Cb^d$ is said to be a \emph{holomorphic homogeneous regular domain (HHR-domain)}  if there exists $s>0$ such that: for every $z \in \Omega$ there exists a holomorphic embedding $\varphi_z: \Omega \rightarrow \Cb^d$ with $\varphi_z(z)=0$ and
\begin{align*}
s \Bb \subset \varphi_z(\Omega) \subset \Bb
\end{align*}
where $\Bb \subset \Cb^d$ is the unit ball. 
\end{definition}

\begin{remark}
In the literature, HHR domains are sometimes called domains with the \emph{uniform squeezing property}, see for instance~\cite{Y2009}. 
\end{remark}

Examples of HHR domains include:
\begin{enumerate}
\item The Teichm\"uller space of hyperbolic surfaces with genus $g$ and $n$ punctures (by the Bers embedding, see~\cite{Gard1987}), 
\item Kobayashi hyperbolic convex domains or more generally $\Cb$-convex domains \cite{F1991, KZ2016, NA2017},  
\item bounded domains where $\Aut(\Omega)$ acts co-compactly on $\Omega$, and
\item strongly pseudoconvex domains~\cite{DFW2014,DGZ2016}.
\end{enumerate}

A general result of Yeung implies that Property~\ref{item:bd_sec} holds on any HHR-domain. 

\begin{theorem}[{Yeung~\cite[Theorem 2]{Y2009}}]\label{thm:yeung} If $\Omega \subset \Cb^d$ is a HHR-domain, then the Bergman metric $g_\Omega$ on $\Omega$ is complete, has bounded sectional curvature, and positive injectivity radius.
\end{theorem}


For certain classes of HHR-domains it is possible to verify Property~\ref{item:SBG}. 

\begin{proposition} Suppose $\Omega$ is a domain biholomorphic to either a
\begin{enumerate}
\item a strongly pseudoconvex domain, 
\item a $\Cb$-convex domain (e.g. a convex domain) which is Kobayashi hyperbolic,
\item a bounded homogeneous domain, or
\item the Teichm\"uller space of hyperbolic surfaces with genus $g$ and $n$ punctures,
\end{enumerate} 
then the Bergman metric $g_\Omega$ on $\Omega$ has Property~\ref{item:SBG} and hence $\Omega$ has bounded intrinsic geometry. 
\end{proposition}

\begin{proof}
For strongly pseudoconvex domains, it is possible to show that 
\begin{align*}
\norm{\partial \log \Bf_\Omega(z,z)}_{g_\Omega}
\end{align*}
is uniformly bounded, see for instance~\cite[Proposition 3.4]{D1994}. We will consider the $\Cb$-convex case in Proposition~\ref{prop:C_convex_SBG} below. A stronger form of Property~\ref{item:SBG} for the Bergman metric on a homogeneous domain was established by Kai-Ohsawa~\cite{KO2007}.

Let $\Tc_{g,n}$ be the  Teichm\"uller space of hyperbolic surfaces with genus $g$ and $n$ punctures. McMullen~\cite{M2000} constructed a $(1,0)$-form $\theta_{1/\ell}$ on $\Tc_{g,n}$ such that $\omega:= \bar{\partial} \theta_{1/\ell}$ is a complete K\"ahler metric, $\norm{\theta_{1/\ell}}_{\omega}$ is uniformly bounded, and $\omega$ is uniformly bi-Lipschitz to the Kobayashi metric. On $\Tc_{g,n}$ the Kobayashi and Bergman metrics are uniformly bi-Lipschitz (see~\cite{C2004,LSY2004a, Y2005}) and so $\omega$ is also uniformly bi-Lipschitz to the Bergman metric. Finally, there exists a smooth function $\lambda: \Tc_{g,n} \rightarrow \Rb$ such that $\partial \lambda = \theta_{1/\ell}$, see for instance~\cite[Section 6.3]{KS2012}. So Property~\ref{item:SBG} holds. 
\end{proof}
%

\section{Preliminaries}

\subsection{Notations} In this section we fix any possibly ambiguous notation. 

\medskip

\noindent \emph{The Bergman metric, kernel, and distance:} We will use the following notations. 

\begin{definition} Suppose $\Omega \subset \Cb^d$ is a pseudoconvex domain. 
\begin{enumerate}
\item Let $\Bf_\Omega$ denote the Bergman kernel on $\Omega$,
\item let $g_\Omega$ denote the Bergman metric on $\Omega$, 
\item let $\dist_\Omega$ denote the distance induced by the Bergman metric, and 
\item for $\zeta \in \Omega$ and $r \geq 0$ let 
\begin{align*}
B_\Omega(\zeta; r) :=\{ z \in \Omega : \dist_\Omega(z,\zeta) < r\}
\end{align*}
denote the open ball of radius $r$ centered at $\zeta$ in the Bergman distance. 
\end{enumerate}
\end{definition} 

\medskip

\noindent \emph{Approximate inequalities:} Given functions $f,h : X \rightarrow \Rb$ we write $f \lesssim h$ or equivalently $h \gtrsim f$ if there exists a constant $C > 0$ such that $f(x) \leq C h(x)$ for all $x \in X$. Often times the set $X$ will be a set of parameters (e.g. $m \in \Nb$). 

\medskip

\noindent \emph{The Levi form: } Given a domain $\Omega \subset \Cb^d$ and a $\Cc^2$-smooth real valued function $f : \Omega \rightarrow \Rb$, the \emph{Levi form of $f$} is 
\begin{align*}
\Lc(f) = \sum_{1 \leq i,j \leq d} \frac{\partial^2 f}{\partial z_i \partial \bar{z}_j}dz_i d\bar{z}_j. 
\end{align*}
Notice that $f$ is plurisubharmonic if $\Lc(f) \geq 0$ and, by definition,
\begin{align*}
\Lc\left(\log \Bf_\Omega(z,z)\right)=g_\Omega.
\end{align*}

\medskip

\noindent \emph{Norms and inner products on $(p,q)$-forms:} Given a $(p,q)$-form $\alpha = \sum \alpha_{I,J} dz^I \wedge d\bar{z}^J$ on a domain $\Omega$, we will let $\norm{\alpha}$ denote the function 
\begin{align*}
z \in \Omega \rightarrow \left( \sum \abs{\alpha_{I,J}(z)}^2 \right)^{1/2}.
\end{align*}
Similarly, we will let $\ip{\cdot, \cdot}$ denote the pointwise inner product on $(p,q)$-forms, that is
\begin{align*}
\ip{ \sum \alpha_{I,J} dz^I \wedge d\bar{z}^J, \sum \beta_{I,J} dz^I \wedge d\bar{z}^J} = \sum \alpha_{I,J}\bar{\beta}_{I,J}.
\end{align*}
So $\norm{\alpha} = \sqrt{ \ip{\alpha,\alpha}}$. Finally, we will use 
\begin{align*}
\norm{\alpha}_\Omega : = \left( \int_\Omega \norm{\alpha}^2 dz \right)^{1/2}
\end{align*}
to denote the norm on $L^2_{(p,q)}(\Omega)$.

\subsection{A sufficient condition for compactness} In this section we recall McNeal's sufficient condition for compactness.

\begin{definition} Suppose $\Omega \subset \Cb^d$ is a domain. A plurisubharmonic function $\Cc^2$ function $\lambda : \Omega \rightarrow \Rb$ has \emph{self bounded complex gradient} if there exists $C > 0$ such that
\begin{align}
\label{eq:SBG}
\abs{\sum_{j=1}^d \frac{\partial \lambda}{\partial z_j} \xi_j }^2 \leq C\Lc(\lambda)(\xi,\xi)
\end{align}
for all $\xi \in \Cb^d$. 
\end{definition}

This can be rephrased as follows: given a $\Cc^2$ plurisubharmonic function $\lambda$ the Levi form $\Lc(\lambda)$ induces a (possibly infinite valued) norm on 1-forms defined by
\begin{align*}
\norm{ \alpha}_{\Lc(\lambda)} = \max\left\{  \abs{\alpha(X)} : X \in \Cb^d, \ \Lc(\lambda)(X,X) \leq 1\right\}.
\end{align*}
Then Equation~\eqref{eq:SBG} is equivalent to $\norm{\partial\lambda}_{\Lc(\lambda)} \leq \sqrt{C}$. We also note that if $t > 0$ and $\lambda_t = t\lambda$, then $\norm{\partial \lambda_t}_{\Lc(\lambda_t)} = t^{-1/2}\norm{\partial\lambda}_{\Lc(\lambda)}$. 

\begin{definition}[{McNeal~\cite{M2002}}] Suppose $1 \leq q \leq d$. A domain $\Omega \subset \Cb^d$ satisfies condition $(\wt{P}_q)$ if for each $M > 0$ there exists a $\Cc^2$ plurisubharmonic function $\lambda=\lambda_M: \Omega \rightarrow \Rb$ with
\begin{enumerate}
\item $\norm{\partial\lambda}_{\Lc(\lambda)} \leq 1$,
\item $\sigma_{d-q+1}(\Lc(\lambda)) \geq M$ outside a compact set of $\Omega$.
\end{enumerate}
\end{definition}

\begin{remark} In the second part of the definition, we are identifying $\Lc(\lambda)$ with the $d$-by-$d$ matrix $\left[ \frac{\partial^2 \lambda}{\partial z_i \partial \bar{z}_j}\right]$ and $\sigma_j(\Lc(\lambda))$ is the $j^{th}$ largest singular value of this matrix. 
 \end{remark}

\begin{theorem}[{McNeal~\cite[Corollary 4.2]{M2002}}]\label{thm:McNeal} If $\Omega \subset \Cb^d$ is a bounded pseudoconvex domain satisfying condition $(\wt{P}_q)$, then the operator $N_q$ is compact. 
\end{theorem}

Condition $(\wt{P}_q)$ is a generalization of Catlin's condition $(P_q)$ where the estimate  $\norm{\partial\lambda}_{\Lc(\lambda)} \leq 1$ is replaced by $\abs{\lambda} \leq 1$, see~\cite{Cat1984b,M2002} for more detail. We also refer the reader to~\cite{Sib1987} for additional details about domains satisfying condition $(P_1)$.

\subsection{Solutions to $\bar{\partial}$}

We will use the following existence theorem for solutions to $\bar{\partial}$.

\begin{theorem}\label{thm:existence} Suppose $\Omega \subset \Cb^d$ is a bounded pseudoconvex domain, $\lambda_1 : \Omega \rightarrow \Rb$ has self bounded complex gradient, and $\lambda_2: \Omega \rightarrow \{-\infty\}\cup \Rb$ is plurisubharmonic. There exists $C > 0$ which only depends on 
\begin{align*}
\sup_{z \in \Omega} \norm{\partial \lambda_1}_{\Lc(\lambda_1)}
\end{align*}
such that: if $\alpha \in L^{2,{\rm loc}}_{(0,1)}(\Omega)$ and $\bar{\partial}\alpha= 0$, then there exists $u \in L^{2,{\rm loc}}(\Omega)$ with $\bar{\partial}u = \alpha$ and 
\begin{align*}
\int_\Omega \abs{u}^2 e^{-\lambda_2} dz \leq C\int_\Omega \norm{\alpha}_{\Lc(\lambda_1)}^2 e^{-\lambda_2} dz
\end{align*}
assuming the right hand side is finite. 
\end{theorem}

A proof of Theorem~\ref{thm:existence} can be found in~\cite[Theorem 4.5 and Section 4.6]{MV2015}. A special case was established earlier in~\cite[Proposition 3.3]{M2001} with essentially the same argument. 

\subsection{Bounded geometry}

In this section we recall a recent result of D. Wu and S.T. Yau involving K\"ahler manifolds with bounded geometry (in the sense of S.Y. Cheng and S.T. Yau in~\cite{CY1980}).

\begin{definition}\label{defn:quasi_bd_geom} A $d$-dimensional K{\"a}hler manifold $(M,g)$ is said to have \emph{bounded geometry}, if there exist constants $r_2 > r_1 > 0$, $C>1$, and a sequence $(A_q)_{q \in \Nb}$ of positive numbers such that: for every point $m \in M$ there is a domain $U \subset \Cb^n$ and a holomorphic embedding $\psi:U \rightarrow M$ satisfying the following properties:
\begin{enumerate}
\item $\psi(0)=m$, 
\item $r_1\Bb \subset U \subset r_2\Bb$,
\item $C^{-1} g_{\Euc} \leq \psi^* g \leq Cg_{\Euc}$, 
\item for every integer $q \geq 0$ 
\begin{align*}
\sup_{x \in U} \abs{ \frac{ \partial^{\abs{\mu}+\abs{\nu}} ((\psi^*g)_{i\bar{j}})}{\partial z^\mu \partial \bar{z}^{\nu}}(x)} \leq A_q \text{ for all } \abs{\mu}+\abs{\nu} \leq q, \ 1 \leq i,j \leq d.
\end{align*}
where  $(\psi^* g)_{i\bar{j}}$ is the component of $\psi^*g$ in terms of the canonical coordinates $z=(z_1,\dots, z_d)$ on $\Cb^d$ and $\mu,\nu$ are multiple indices with $\abs{\mu}=\mu_1+\dots+\mu_d$.
\end{enumerate}
\end{definition}

We will use the following theorem of D. Wu and S.T. Yau.

\begin{theorem}[{Wu-Yau~\cite[Theorem 9]{WY2020}}]\label{thm:nec_suff_quasi_bd_geom} Let $(M,g)$ be a complete K{\"a}hler manifold of complex dimension $d$. The manifold $(M,g)$ has quasi-bounded geometry if and only if $(M,g)$  has positive injectivity radius and for every integer $q \geq 0$, there exists a constant $C_q > 0$ such that the curvature tensor $R$ of $g$ satisfies
\begin{align*}
\sup_{M} \norm{ \nabla^q R}_{g}  \leq C_q.
\end{align*}
Moreover, one can choose the constants $r_1$, $r_2$, $C$, $( A_q)_{q \geq 0}$ in Definition~\ref{defn:quasi_bd_geom} to depend only on $\{ C_q\}_{q \geq 0}$ and $d$. 
\end{theorem}

\section{The (complex) convex case}\label{sec:the_convex_case}

The primary purpose of this section is to verify that the Bergman metric on a convex domain or more generally a $\Cb$-convex domain satisfies Definition~\ref{defn:BBG}. By Theorem~\ref{thm:yeung}, it is enough to verify that the Bergman metric satisfies Property~\ref{item:SBG}.

We will also provide a proof of Theorem~\ref{thm:compactness_intro} in the special case of convex domains. In this case the proof is similar to the argument for general domains with bounded intrinsic geometry, but has less technicalities. 

\subsection{The convex case}

A domain $\Omega \subset \Cb^d$ is called \emph{$\Cb$-properly convex} if it is convex and every complex affine map $\Cb \rightarrow \Omega$ is constant. By a result of Barth, a convex domain is Kobayashi hyperbolic if and only if it is $\Cb$-properly convex~\cite{B1980}. 

The key tool in the convex case is a result of Frankel which says that any $\Cb$-properly convex domain can be normalized via an affine map. In what follows, we will let $\Aff(\Cb^d)$ denote the group of affine automorphisms of $\Cb^d$. Any $T \in \Aff(\Cb^d)$ can be written as $T(z) = b + L z$ where $b \in \Cb^d$ and $L \in \GL_d(\Cb)$. Then the matrix $L$ is called the \emph{linear part} of $T$.  

\begin{theorem}[Frankel~\cite{F1991}]\label{thm:frankel_i} For any $d \in \Nb$ there exists $\epsilon_d > 0$ such that: if $\Omega \subset \Cb^d$ is a $\Cb$-properly convex domain and $\zeta \in \Omega$, then there exists $T_\zeta \in \Aff(\Cb^d)$ with $T_\zeta(\zeta) = 0$ and 
\begin{align*}
2\epsilon_d\Bb \subset T_\zeta(\Omega) \subset  \Hc^d
\end{align*}
where $\Hc = \{ z \in \Cb : {\rm Im}(z) > -1\}$. 
\end{theorem}

\begin{remark} Notice that this implies that every $\Cb$-properly convex domain is an HHR-domain. \end{remark}

Frankel used the normalizing maps to estimate the Bergman metric in terms of the Euclidean geometry of the domain. Given a domain $\Omega \subset \Cb^d$,  $z \in \Omega$, and $v \in \Cb^d$ non-zero define
 \begin{align*}
 \delta_\Omega(z;v) = \min\{ \norm{w-z} : w \in \partial \Omega \cap (z + \Cb \cdot v)\}.
 \end{align*}

 \begin{theorem}[Frankel~\cite{F1991}]\label{thm:frankel_ii} For any $d \in \Nb$ there exists $A_d > 1$ such that: if $\Omega \subset \Cb^d$ is a $\Cb$-properly convex domain, then
 \begin{align*}
 \frac{1}{A_d}\frac{\norm{v}}{\delta_\Omega(z;v)} \leq \sqrt{g_{\Omega,z}(v,v)} \leq  A_d\frac{\norm{v}}{\delta_\Omega(z;v)}
 \end{align*}
 for all $z \in \Omega$ and non-zero $v \in \Cb^d$. 
 \end{theorem} 

\medskip

\textbf{Standing assumption:} For the rest of this section let $\Omega \subset \Cb^d$ be a properly convex domain and for each $\zeta \in \Omega$ let $T_\zeta$ be an affine map satisfying Theorem~\ref{thm:frankel_i}. 

\medskip

We will show that the Bergman metric on $\Omega$ has Property~\ref{item:SBG} and then prove Theorem~\ref{thm:compactness_intro} for $\Omega$. 

\begin{lemma}\label{lem:convex_basic_i} \ \begin{enumerate}
\item There exists $C_1> 1$ such that 
\begin{align*}
\frac{1}{C_1} \leq \Bf_{T_{\zeta}(\Omega)}(w,w) \leq C_1
\end{align*}
for all $\zeta \in \Omega$ and $w \in \epsilon_d\Bb$. 
\item For all multi-indices $a,b$ there exists $C_{a,b} > 0$ such that 
\begin{align*}
\frac{\partial^{\abs{a}+\abs{b}} \Bf_{T_{\zeta}(\Omega)}}{\partial u^{a}\partial \bar{w}^b}(u,w) \leq C_{a,b}
\end{align*}
for all $\zeta \in \Omega$ and $u,w \in \epsilon_d\Bb$. 
\end{enumerate}
\end{lemma}

\begin{proof} Fix some $\delta \in (\epsilon_d,2\epsilon_d)$. From the monotoncity property of the Bergman kernel and the explicit formulas for the Bergman kernel on $2\epsilon_d\Bb$ and $\Hc^d$, there exists $C> 1$ such that 
\begin{align*}
\frac{1}{C} \leq \Bf_{T_{\zeta}(\Omega)}(w,w) \leq C
\end{align*}
for all $\zeta \in \Omega$ and $w \in \delta\Bb$. 

For Part (2), notice that 
\begin{align*}
\abs{ \Bf_{T_{\zeta}(\Omega)}(u,w)} \leq \sqrt{ \Bf_{T_{\zeta}(\Omega)}(u,u)}\sqrt{ \Bf_{T_{\zeta}(\Omega)}(w,w)} \leq C. 
\end{align*}
on $\delta\Bb \times \delta\Bb$. Further, $\Bf_{T_{\zeta}(\Omega)}$ is holomorphic in the first variable and anti-holomorphic in the second variable. Then, since $\delta > \epsilon_d$,  Cauchy's integral formulas imply uniform estimates for the derivates on  $\epsilon_d\Bb \times \epsilon_d\Bb$.
\end{proof}

We will also use the following corollary to Theorem~\ref{thm:frankel_ii}. 

\begin{corollary}[to Theorem~\ref{thm:frankel_ii}]\label{cor:convex_basic_ii} There exists $C_2> 1$ such that 
\begin{align*}
\frac{1}{C_2}\norm{X} \leq \sqrt{g_{T_\zeta(\Omega),w}(X,X)}\leq C_2 \norm{X}
\end{align*}
for all $\zeta \in \Omega$, $w \in \epsilon_d \Bb$, and $X \in \Cb^d$. 
\end{corollary}

Using these estimates we can prove that the Bergman metric on a convex domain satisfies Property~\ref{item:SBG}.

\begin{proposition}\label{prop:SBG_convex} $\norm{\partial \log \Bf_\Omega(z,z)}_{g_\Omega}$ is uniformly bounded and hence $g_\Omega$ has Property~\ref{item:SBG}. \end{proposition}

 \begin{proof} Fix $\zeta \in \Omega$. Notice that 
 \begin{align*}
\Bf_{\Omega}(z,z)=\Bf_{T_\zeta(\Omega)}(T_\zeta(z),T_\zeta(z))\abs{\det(L_\zeta)}^2
 \end{align*}
 where $L_\zeta$ is the linear part of $T_\zeta$. So by Lemma~\ref{lem:convex_basic_i} and Corollary~\ref{cor:convex_basic_ii} 
  \begin{align*}
 \abs{\partial \log \Bf_{\Omega}(z,z)(X)}_{z=\zeta} &= \abs{\partial \log \Bf_{T_\zeta(\Omega)}(w,w)(L_\zeta X)}_{w=0} \\
 & \lesssim  \norm{L_\zeta X} \lesssim \sqrt{g_{T_\zeta(\Omega),0}(L_\zeta X,L_\zeta X)} \\
 & = \sqrt{g_{\Omega,w}(X,X)}
 \end{align*}
 for all $X \in \Cb^d$. So $\norm{\partial \log \Bf_\Omega(z,z)}_{g_\Omega}$ is uniformly bounded.
 \end{proof}

 Finally we provide a proof of Theorem~\ref{thm:compactness_intro} for the special case of convex domains. As mentioned at the start of this section, the proof in this case is similar to the proof in the general case (and also similar to Fu-Straube's original proof), but in this special case many technicalities can be avoided. 
 
 \begin{theorem}[Fu-Straube~\cite{FS1998}] Suppose that $\Omega$ is bounded (recall, we have already assumed that $\Omega$ is convex). Then the following are equivalent:
\begin{enumerate}
\item $N_q$ is compact.
\item If $g_{\Omega,z}$ is identified with the matrix $\left[ g_{\Omega,z}(\frac{\partial}{\partial z_i}, \frac{\partial}{\partial \bar{z}_j})\right]$, then 
\begin{align*}
\lim_{z\rightarrow \partial\Omega} \sigma_{d-q+1}( g_{\Omega,z})  =\infty.
\end{align*}
\item $\partial \Omega$ contains no $q$-dimensional analytic varieties.
\end{enumerate}
\end{theorem}

\begin{remark}
The proof below is similar to Fu and Straube's original argument that $(1) \Leftrightarrow (3)$, but with three modifications that will allow us to extend the result to domains with bounded intrinsic geometry. 
\begin{itemize}
\item The first is the observation that the estimates in Theorem~\ref{thm:frankel_ii} imply that $(2) \Leftrightarrow (3)$. This allows us to work with the Bergman metric instead of the boundary of the domain. 
\item In their proof that $(2/3) \Rightarrow (1)$, Fu and Straube directly construct bounded plurisubharmonic functions which satisfy Catlin's property $(P_q)$. This construction seems to rely on the convexity of the domain. In contrast, we will use Proposition~\ref{prop:SBG_convex} which directly shows that (2) implies property $(\wt{P}_q)$ and hence compactness. 
\item In their proof that $(1) \Rightarrow (2/3)$, Fu and Straube consider a linear slice of the convex domain and use the Ohsawa-Takegoshi extension theorem to pass from the slice to the full domain. The fact that linear slices are well behaved again seems to rely on the convexity of the domain. Our argument that $(1) \Rightarrow (2/3)$ is similar, but by using Frankel's normalizing maps we can avoid this reduction to a lower dimensional domain. 
\end{itemize}
\end{remark}

\begin{proof}Theorem~\ref{thm:frankel_ii} implies that $(2) \Leftrightarrow (3)$. If (2) is true, then Proposition~\ref{prop:SBG_convex} implies that $\Omega$ has Property $(\wt{P}_q)$ and hence $N_q$ is compact by Theorem~\ref{thm:McNeal}.

We prove that $(1) \Rightarrow (2)$ by contradiction. Suppose for a contradiction that $(1)$ is true and $(2)$ is false. Then there exist $C_3 > 0$, a sequence $(\zeta_m)_{m \geq 1}$ in $\Omega$ converging to $\partial \Omega$, and a sequence $(V_m)_{m \geq 1}$ of $q$-dimensional linear subspaces such that 
\begin{align}
\label{eq:bd_on_bergman_metric_convex}
\sqrt{g_{\Omega,\zeta_m}(v,v)} \leq C_3 \norm{v}
\end{align}
for all $v \in V_m$. 

For each $m$, let $T_m = T_{\zeta_m}$ be the recentering map from Theorem~\ref{thm:frankel_i} and let $L_m$ be the linear part of $T_m$. Corollary~\ref{cor:convex_basic_ii} and Equation~\eqref{eq:bd_on_bergman_metric_convex} imply that
\begin{align*}
\norm{L_m v} \leq C_2\sqrt{g_{T_m(\Omega),0}(L_mv,L_mv)} =C_2 \sqrt{g_{\Omega,\zeta_m}(v,v)} \leq C_2C_3 \norm{v}
\end{align*}
for all $v \in V_m$. Since $\Omega$ is bounded, Theorem~\ref{thm:frankel_ii} implies that there exists $C_4 > 1$ such that $g_{\Omega} \geq C_4^{-2} g_{\Euc}$. Then 
\begin{align*}
\norm{L_m v} \geq \frac{1}{C_2} \sqrt{g_{T_m(\Omega),0}(L_mv,L_mv)} =\frac{1}{C_2} \sqrt{g_{\Omega,\zeta_m}(v,v)} \geq \frac{1}{C_2C_4} \norm{v}
\end{align*}
 for all $v \in \Cb^d$. So there exists $C_5 > 1$ such that
\begin{align}\label{eq:sing_values}
\frac{1}{C_5} \leq \sigma_{j}(L_m) \leq C_5
\end{align}
for all $d-q+1 \leq j \leq d$. 
 
Using the singular value decomposition we can write $L_m^{-1} = k_{1,m}D_m k_{2,m}$ where $k_{1,m}, k_{2,m}$ are unitary matrices and 
\begin{align*}
D_m =  \begin{pmatrix} \sigma_d(L_m)^{-1} & & \\ & \ddots & \\ & & \sigma_1(L_m)^{-1} \end{pmatrix}.
\end{align*}
Then consider the $(0,q)$-form 
\begin{align*}
\alpha_m = \frac{\Bf_\Omega(\cdot, \zeta_m)}{\sqrt{\Bf_\Omega(\zeta_m, \zeta_m)}} (k^{-1}_{1,m})^* d\bar{z}_1 \wedge \dots \wedge d\bar{z}_q
\end{align*}
on $\Omega$. Then $\norm{\alpha_m}_\Omega=1$ and $\bar{\partial} \alpha_m= 0$. So $h_m : = \bar{\partial}^* N_q \alpha_m$ satisfies $\bar{\partial} h_m = \alpha_m$ and $\{ h_m : m \geq 1\}$ is relatively compact in $L^2_{(0,q-1)}(\Omega)$ (see the discussion proceeding Theorem~\ref{thm:compactness}). 

By passing to a subsequences we can suppose that $h_m$ converges in $L^2_{(0,q-1)}(\Omega)$. Then for any $\epsilon > 0$ there exists a compact subset $K \subset \Omega$ such that 
\begin{align}
\label{eq:unif_estimate_convex}
\sup_{m \geq 0} \int_{\Omega \setminus K} \norm{h_m}^2 d\mu < \epsilon.
\end{align}
We will derive a contradiction by showing that 
\begin{align*}
\int_{B_\Omega(\zeta_m;r)} \norm{h_m}^2 dz
\end{align*}
is uniformly bounded from below. Since $\zeta_m \rightarrow \partial \Omega$ and the Bergman metric is complete, this will contradict Equation~\eqref{eq:unif_estimate_convex}.

Consider the $(0,q)$-form on $T_m(\Omega)$ defined by
\begin{align*}
\wt{\alpha}_m 
&= \det(L_m^{-1}) (T_m^{-1})^*\alpha_m =   \det(L_m^{-1}) \frac{\Bf_\Omega(T_m^{-1}(\cdot), \zeta_m)}{\sqrt{\Bf_\Omega(\zeta_m, \zeta_m)}} (L_m^{-1})^*(k^{-1}_{1,m})^* d\bar{z}_1 \wedge \dots \wedge d\bar{z}_q \\
& = J_m\frac{\Bf_{T_m(\Omega)}(\cdot, 0)}{\sqrt{\Bf_{T_m(\Omega)}(0,0)}} (k_{2,m})^* d\bar{z}_1 \wedge \dots \wedge d\bar{z}_q
\end{align*} 
where $J_m = \prod_{j=0}^{q-1} \sigma_{d-j}(L_m)^{-1}$. Using Lemma~\ref{lem:convex_basic_i} and Equation~\eqref{eq:sing_values}, we can pass to a subsequence such that $\wt{\alpha}_m$ converges uniformly on $\epsilon_d\Bb$ to a smooth $(0,q)$-form $\wt{\alpha}$ with $\wt{\alpha}|_0 \neq 0$. 

Since $\wt{\alpha} \neq 0$, there exists a smooth compactly supported $(0,q)$-form $\chi : \epsilon_d \Bb \rightarrow \Cb$ such that 
\begin{align*}
0 < \int_{\epsilon_d \Bb} \ip{\wt{\alpha}, \chi} dw.
\end{align*}
Next, notice that $\wt{\alpha}_m = \det(L_m^{-1}) (T_m^{-1})^*\bar{\partial}h_m = \det(L_m^{-1}) \bar{\partial}(T_m^{-1})^*h_m$ and so
\begin{align*}
\int_{\epsilon_d \Bb} \ip{\wt{\alpha}, \chi} & dw = \lim_{m \rightarrow \infty} \int_{\epsilon_d \Bb} \ip{\wt{\alpha}_m, \chi} dw  =  \lim_{m \rightarrow \infty} \det(L_m^{-1}) \int_{\epsilon_d \Bb} \ip{ \bar{\partial}(T_m^{-1})^*h_m, \chi} dw \\
& = \lim_{m \rightarrow \infty} \int_{\epsilon_d \Bb} \ip{ \det(L_m^{-1}) (T_m^{-1})^*h_m, \vartheta\chi} dw
\end{align*}
where $\vartheta$ is the formal adjoint of $\bar{\partial}$. By Cauchy Schwarz and Equation~\eqref{eq:sing_values}
\begin{align*}
\int_{\epsilon_d \Bb} &  \ip{\det(L_m^{-1}) (T_m^{-1})^*h_m, \vartheta\chi} dw \lesssim \left(\int_{\epsilon_d \Bb} \abs{\det(L_m^{-1}) }^2 \norm{(T_m^{-1})^*h_m}^2 dw \right)^{1/2} \\
& \leq \norm{L_m^{-1}}^{q-1} \left(\int_{\epsilon_d \Bb} \abs{\det(L_m^{-1}) }^2 \norm{h_m|_{T_m^{-1}(w)}}^2 dw\right)^{1/2} \\
& = \frac{1}{\sigma_d(L_m)^{q-1}} \left( \int_{T_{m}^{-1}(\epsilon_d \Bb)} \norm{h_m}^2 dz \right)^{1/2} \\
& \lesssim  \left( \int_{T_{m}^{-1}(\epsilon_d \Bb)} \norm{h_m}^2 dz \right)^{1/2}.
\end{align*}
By Corollary~\ref{cor:convex_basic_ii} we have $T_{m}^{-1}(\epsilon_d \Bb) \subset B_\Omega(\zeta_m; C_2 \epsilon_d)$ and so 
\begin{align*}
0 <\liminf_{m \rightarrow \infty}  \left( \int_{B_\Omega(\zeta_m;r)} \norm{h_m}^2 dz \right)^{1/2}
\end{align*}
for any $r > C_2\epsilon_d$. Thus we have a contradiction. 
\end{proof}

\subsection{The $\Cb$-convex case}

A domain $\Omega \subset \Cb^d$ is called \emph{$\Cb$-convex} if for every complex affine line $L \subset \Cb^d$ the intersection $\Omega \cap L$ is either empty or simply connected. Clearly, every convex domain is $\Cb$-convex. Further, as in the convex case, we say that a domain is a \emph{$\Cb$-properly $\Cb$-convex domain} if it is $\Cb$-convex and every complex affine map $\Cb \rightarrow \Omega$ is constant. As in the convex case, a $\Cb$-convex domain is Kobayashi hyperbolic if and only if it is $\Cb$-properly $\Cb$-convex, see for instance~\cite{NPZ2011}.

For $\Cb$-convex domains, we have the following recentering result established by Nikolov-Andreev using results from~\cite{NPZ2011}. 

\begin{theorem}[{Nikolov-Andreev~\cite[proof of Theorem 1]{NA2017}}]\label{thm:C_convex_recentering} For any $d \in \Nb$ there exists $\epsilon_d > 0$ such that: if $\Omega \subset \Cb^d$ is a $\Cb$-properly $\Cb$-convex domain and $\zeta \in \Omega$, then there exists  $T_\zeta \in \Aff(\Cb^d)$ such that $T_\zeta(\zeta) = 0$ and 
\begin{align*}
2\epsilon_d\Bb \subset T_\zeta(\Omega) \subset  \prod_{j=1}^d D_j
\end{align*}
where each $D_j \subset \Cb$ is a simply connected domain with $\dist_{\Euc}(0,\partial D_j) \leq 1$. 
\end{theorem}

To show that the Bergman metric satisfies Property~\ref{item:SBG}, we will need the following estimates.

\begin{theorem}[{Nikolov-Pflug-Zwonek~\cite[Proposition 1, Theorem 12]{NPZ2011}}] For any $d \in \Nb$ there exists $A_d > 1$ such that: if $\Omega \subset \Cb^d$ is a $\Cb$-properly $\Cb$-convex domain, then
 \begin{align*}
 \frac{1}{A_d}\frac{\norm{v}}{\delta_\Omega(z;v)} \leq \sqrt{g_{\Omega,z}(v,v)} \leq  A_d\frac{\norm{v}}{\delta_\Omega(z;v)}
 \end{align*}
 for all $z \in \Omega$ and non-zero $v \in \Cb^d$. 
 \end{theorem}

\begin{lemma}\label{lem:Bergman_planar_bd} If $D\subsetneq \Cb$ is simply connected, then 
\begin{align*}
\frac{1}{16\delta_D(z)^2} \leq \Bf_D(z,z) \leq \frac{1}{\delta_D(z)^2}
\end{align*}
where $\delta_D(z) = \inf\{ \abs{w-z} : w \in \Cb \setminus D\}$.
\end{lemma}

\begin{proof} Fix $z \in D$ and let $\psi: D \rightarrow \Db$ be a biholomorphism with $\psi(z)=0$. Then $\Bf_D(z,z) = \Bf_{\Db}(0,0) \abs{\psi^\prime(z)}^2 = \abs{\psi^\prime(z)}^2$. The Koebe 1/4 theorem applied to $\psi^{-1}$ says that 
\begin{align*}
4\delta_D(z) \geq \abs{(\psi^{-1})^\prime(0)} = \frac{1}{\abs{\psi^\prime(z)}}
\end{align*}
and so $\Bf_D(z,z) \geq \frac{1}{16}\delta_D(z)^{-2}$. Applying the Schwarz lemma to $w \in \Db \rightarrow \psi(\delta_D(z)w)$ shows that $\abs{\psi^\prime(z)} \leq \delta_D(z)^{-1}$ and so $\Bf_D(z,z) \leq \delta_D(z)^{-2}$.
\end{proof}

\begin{proposition}\label{prop:C_convex_SBG} If $\Omega \subset \Cb^d$ is a $\Cb$-proper $\Cb$-convex domain, then 
\begin{align*}
\sup_{z \in \Omega} \norm{\partial \log \Bf_\Omega(z,z)}_{g_\Omega} < +\infty.
\end{align*}
Hence $g_\Omega$ has Property~\ref{item:SBG}. 
\end{proposition}

\begin{proof} For each $\zeta \in \Omega$, fix  $T_\zeta \in \Aff(\Cb^d)$ an affine map satisfying Theorem~\ref{thm:C_convex_recentering}. Fix $\delta \in (\epsilon_d,2\epsilon_d)$. Using Lemma~\ref{lem:Bergman_planar_bd} there exists $A > 1$ such that 
\begin{align*}
\frac{1}{A} \leq \Bf_{T_\zeta(\Omega)}(w,w) \leq A
\end{align*}
for all $\zeta \in \Omega$ and $w \in \delta\Bb$. Then using Cauchy's integral formulas and increasing $A$ one can prove that 
\begin{align*}
\abs{\partial\log \Bf_{T_\zeta(\Omega)}(w,w)(X) } \leq A \norm{X}
\end{align*}
for all $\zeta \in \Omega$, $w \in \epsilon_d\Bb$, and $X \in \Cb^d$. Then the rest of the proof is identical to the proof of Proposition~\ref{prop:SBG_convex}. 
\end{proof}

\section{Local charts from bounded geometry}\label{sec:charts}

The following constructions are fundamental for everything else in the paper. 

\begin{theorem}\label{thm:charts} Suppose $\Omega \subset \Cb^d$ is a domain, $g$ is a complete K\"ahler metric on $\Omega$, and $\dist_g$ is the distance induced by $g$.
\begin{enumerate}
\item If $g$ has Property~\ref{item:bd_sec}, then there exists $ A_1>1$ such that: For every $\zeta \in \Omega$ there exists a holomorphic embedding $\Phi_\zeta : \Bb \rightarrow \Omega$ with  $\Phi_\zeta(0) = \zeta$,
\begin{align*}
\frac{1}{A_1} g_{\Euc} \leq \Phi_\zeta^* g \leq A_1 g_{\Euc},
\end{align*}
and 
\begin{align*}
\frac{1}{\sqrt{A_1}}\norm{w-u} \leq \dist_g( \Phi_\zeta(w),\Phi_\zeta(u)) \leq \sqrt{A_1}\norm{w-u}.
\end{align*}
\item  If $g$ has Property~\ref{item:SBG} and $r > 0$, then there exists $ A_2=A_2(r)>0$ such that: For every $\zeta \in \Omega$ there exists a plurisubharmonic function $\phi_\zeta : \Omega \rightarrow \Rb$ with
\begin{enumerate}
\item $\Lc(\phi_\zeta) \geq g$ on $B_g(\zeta; r):=\{z \in \Omega : \dist_g(z,\zeta) < r\}$,
\item $-A_2 \leq \phi_\zeta \leq  0$ on $\Omega$. 
\end{enumerate}
\end{enumerate}
\end{theorem} 

For the rest of the section let $\Omega \subset \Cb^d$ be a domain and let $g$ be a complete K\"ahler metric on $\Omega$. 

\subsection{Part (1)} We will show that part (1) is a consequence of deep results of Shi~\cite{Shi1989} and Wu-Yau~\cite{WY2020}.  

Suppose $g$ has Property~\ref{item:bd_sec}. Since $g$ is complete and has bounded sectional curvature, by a result of Shi~\cite{Shi1989} there exist $C_0 > 1$ and a complete K\"ahler metric $h$ on $\Omega$ such that 
\begin{align*}
\frac{1}{C_0} g \leq h \leq C_0 g
\end{align*}
and for every $q \geq 0$ 
\begin{align*}
\sup_{z \in \Omega} \norm{\nabla^q R(h)}_{h}  < + \infty
\end{align*}
where $R(h)$ is the curvature tensor of $h$ (this metric is obtained by applying the Ricci flow to $g$ for a small amount of time). 

\begin{lemma} $h$ has positive injectivity radius. \end{lemma}

\begin{proof} 

Since $g$ has bounded sectional curvature and positive injectivity radius, the Rauch comparison theorem implies that there exists $r_1 > 0$ and $C_1 > 1$ such that 
\begin{align*}
F_\zeta: = \exp_{\zeta}^{-1}|_{B_g(\zeta;r_1)} : B_g(\zeta;r_1) \rightarrow r_1 \Bb
\end{align*}
is a well defined diffeomorphism and 
\begin{align*}
\frac{1}{C_1} g \leq F_\zeta^* g_{\Euc} \leq C_1 g
\end{align*}
for every $\zeta \in \Omega$, see for instance~\cite[Section 8.7]{G2007}. Then 
\begin{align*}
\frac{1}{C_0C_1} h \leq F_\zeta^* g_{\Euc} \leq C_0C_1 h
\end{align*}
and so~\cite[Proposition 2.1]{LSY2004b} implies that $h$ has positive injectivity radius. 
\end{proof}

Now applying Theorem~\ref{thm:nec_suff_quasi_bd_geom} to the K\"ahler manifold $(\Omega,h)$ yields constants $C_2>1$, $r_1,r_2 > 0$, and holomorphic embeddings $F_\zeta : U_\zeta \rightarrow \Omega$ such that $F_\zeta(0)=\zeta$, 
\begin{align*}
\frac{1}{C_2} g_{\Euc} \leq F_\zeta^* h \leq C_2 g_{\Euc},
\end{align*}
and $r_1 \Bb \subset U_\zeta \subset r_2 \Bb$. 

Let $r := \min\{ \frac{r_1}{2}, \frac{1}{2}\}$ and 
\begin{align*}
\Phi_\zeta &: 2\Bb \rightarrow \Omega \\
\Phi_\zeta(w)& = F_\zeta(rw).
\end{align*}
We claim that $\Phi_\zeta|_{\Bb}$ satisfies part (1) of the theorem with $A_1 := 2C_0C_2$. Notice that $(\Phi_\zeta^* g)_z = r (F_\zeta^* g)_{rz}$. So
 \begin{align*}
\frac{1}{A_1} g_{\Euc} \leq \frac{r}{C_0C_2} g_{\Euc} \leq \Phi_\zeta^* g \leq rC_0C_2 g_{\Euc} \leq A_1 g_{\Euc}. 
\end{align*}

To establish the bounds on the distance, recall that the length  of a piecewise $\Cc^1$ curve $\sigma : [a,b] \rightarrow \Omega$ with respect to $g$ is defined by 
\begin{align*}
{\rm length}_g( \sigma)=\int_a^b \sqrt{ g_{\sigma(t)}(\sigma^\prime(t), \sigma^\prime(t))} dt
\end{align*}
and the distance induced by $g$ is defined by 
\begin{align*}
\dist_g(u,w) = \inf {\rm length}_g( \sigma)
\end{align*}
where the infimum is taken over all peicewise $\Cc^1$ curves joining $u$ to $w$. 

Now fix $u,w \in \Bb$. Since $\Phi_\zeta^* g  \leq A_1 g_{\Euc}$, we clearly have 
\begin{align*}
\dist_g( \Phi_\zeta(w),\Phi_\zeta(u)) \leq \sqrt{A_1}\norm{w-u}.
\end{align*}
To establish the lower bound, consider some piecewise $\Cc^1$ curve $\sigma:[0,T] \rightarrow \Omega$  joining $\Phi_\zeta(u)$ to $\Phi_\zeta(w)$. If $\sigma([0,T]) \subset \Phi_\zeta(2\Bb)$, then the estimate $\Phi_\zeta^* g  \geq (A_1)^{-1} g_{\Euc}$ implies that 
\begin{align*}
{\rm length}_g( \sigma) \geq \frac{1}{\sqrt{A_1}}\norm{w-u}.
\end{align*}
Otherwise, there exists sequences $(a_n)_{n \geq 1}, (b_n)_{n \geq 1}$ in $[0,T]$ such that $\sigma([0,a_n]), \sigma([0,b_n]) \subset \Phi_\zeta(2\Bb)$ and 
\begin{align*}
\lim_{n \rightarrow \infty} \norm{ \Phi_\zeta^{-1}(a_n)} = 2 = \lim_{n \rightarrow \infty} \norm{ \Phi_\zeta^{-1}(b_n)}.
\end{align*}
Then 
 \begin{align*}
{\rm length}_g( \sigma)& \geq \limsup_{n \rightarrow \infty} ~ {\rm length}_g\left( \sigma|_{[0,a_n]}\right)+{\rm length}_g\left(\sigma|_{[b_n,T]}\right) \\
& \geq \frac{2}{\sqrt{A_1}} \geq \frac{1}{\sqrt{A_1}} \norm{u-w}.
\end{align*}
Then since $\sigma$ was an arbitrary piecewise $\Cc^1$ curve joining $\Phi_\zeta(u)$ to $\Phi_\zeta(w)$ we have 
\begin{align*}
\dist_g( \Phi_\zeta(w),\Phi_\zeta(u)) \geq \frac{1}{\sqrt{A_1}}\norm{w-u}.
\end{align*}

\subsection{Part (2)} Suppose $g$ has Property~\ref{item:SBG}. Then, by definition, there exist $C > 1$ and a $\Cc^2$ function $\lambda : \Omega \rightarrow \Rb$ such that 
\begin{align*}
\frac{1}{C} g \leq \Lc(\lambda) \leq C g
\end{align*}
 and $\norm{\partial \lambda}_{g} \leq C$.  

We start by observing that the function $\lambda$ can be used to construct negative plurisubharmonic functions. 

\begin{lemma}\label{lem:bd_psh} If $\eta > 0$ is sufficiently small, then $-e^{-\eta \lambda}$ is strictly plurisubharmonic and 
\begin{align*}
\Lc\left( -e^{-\eta \lambda}\right) \geq \frac{\eta}{2C} e^{-\eta \lambda} g. 
\end{align*}
\end{lemma}

\begin{proof} Notice that
\begin{align*}
\Lc( -e^{-\eta \lambda}) = \eta e^{-\eta \lambda} \left( \Lc(\lambda) -   \eta\partial \lambda \otimes \overline{\partial\lambda} \right) \geq  \frac{\eta}{C} e^{-\eta \lambda} \left( g-   C\eta\partial \lambda \otimes \overline{\partial\lambda} \right).
\end{align*}
Then, since $\norm{\partial \lambda}_{g}$ is uniformly bounded, for $\eta$ sufficiently small we have
\begin{equation*}
\Lc\left( \lambda \right) \geq  \frac{\eta}{2C}  e^{-\eta \lambda} g. \qedhere
\end{equation*}
\end{proof}

\begin{lemma} $\abs{\lambda(z)-\lambda(w)} \leq 2C\dist_g(z,w)$. \end{lemma}

\begin{proof} Since $\lambda$ is real valued, $\bar{\partial} \lambda = \overline{\partial \lambda}$ and so $\norm{d\lambda}_{g}=\norm{\partial \lambda+ \bar{\partial} \lambda}_{g} \leq 2C$. 
\end{proof}

Now fix $r > 0$. Combining the last two lemmas, there exists $M =M(r) > 0$ such that: if 
\begin{align*}
\psi_\zeta(z) = -M e^{\eta \left( \lambda(\zeta)-\lambda(z)\right)} 
\end{align*}
then
\begin{align*}
\Lc(\psi_\zeta) \geq g
\end{align*}
on $B_g(\zeta; r)$. Notice that 
\begin{align*}
-M e^{2C\eta r} \leq \psi_\zeta(z) \leq 0 
\end{align*}
on $B_g(\zeta; r)$. Next pick a smooth monotone increasing convex function $\chi: \Rb \rightarrow [0,\infty)$ such that $\chi(t) = 0$ on $(-\infty, -1-M e^{2C\eta r}]$ and $\chi^{\prime}(t) > 1$ on $[-M e^{2C\eta r},\infty)$.  Then 
\begin{align*}
\phi_\zeta = -\chi(0)+\chi \circ \psi_\zeta
\end{align*}
satisfies part (2) since 
\begin{align*}
\Lc(\phi_\zeta) \geq \chi^{\prime}( \psi_\zeta(z)) \Lc(\psi_\zeta) \geq g
\end{align*}
on $B_g(\zeta; r)$.

\section{The pluricomplex Green function}

In this section we establish a local estimate for the pluricomplex Green function on domains with bounded intrinsic geometry. We will use this estimate to study the Kobayashi metrics in Section~\ref{sec:Kobayashi} and to establish an extension result in Section~\ref{sec:extensions}.

\begin{definition} Suppose $\Omega \subset \Cb^d$ is a domain. The \emph{pluricomplex Green function} $\Gf_\Omega(z,u) : \Omega \times \Omega \rightarrow \{-\infty\} \cup (-\infty, 0]$ is defined by 
\begin{align*}
\Gf_\Omega(z,w) = \sup u(z)
\end{align*}
where the supremum is taken over all negative plurisubharmonic functions $u$ such that $u - \log \norm{z-w}$ is bounded from above in a neighborhood of $w$. 
\end{definition}

\begin{remark}
In the definition, we assume that $u \equiv -\infty$ is a plurisubharmonic function. 
\end{remark}

We will frequently use the following basic fact.

\begin{proposition}\cite[Theorem 1.1]{K1985} If $\Omega_1 \subset \Cb^{d_1}$, $\Omega_1 \subset \Cb^{d_1}$, and $f : \Omega_1 \rightarrow \Omega_2$ is a holomorphic map, then
\begin{align*}
\Gf_{\Omega_2}(f(z), f(w)) \leq \Gf_{\Omega_1}(z,w)
\end{align*}
for all $z,w \in \Omega_1$. In particular, if $f$ is a biholomorphism, then $\Gf_{\Omega_2}(f(z), f(w)) = \Gf_{\Omega_1}(z,w)$ for all $z,w \in \Omega_1$. 
\end{proposition}

The main result in this section is the following.

\begin{theorem}\label{thm:green_fcn_bds} Suppose $\Omega \subset \Cb^d$ is a domain with bounded intrinsic geometry, $g$ is a complete K\"ahler metric on $\Omega$ satisfying Definition~\ref{defn:BBG}, and $\dist_g$ is the distance induced by $g$. Then there exist $C, \tau> 0$ such that: 
\begin{align*}
\log \dist_g(z,w) - C \leq  \Gf_\Omega(z,w) \leq \log \dist_g(z,w) + C
\end{align*}
for all $z,w \in \Omega$ with $\dist_g(z,w) \leq \tau$. 
\end{theorem}

\begin{proof} By Theorem~\ref{thm:charts} there exist $A > 1$, holomorphic embeddings $\Phi_\zeta : \Bb \rightarrow \Omega$, and plurisubharmonic functions $\phi_\zeta : \Bb \rightarrow \Omega$ such that 
\begin{enumerate}
\item $\Phi_\zeta(0)=\zeta$,
\item $A^{-1} g_{\Euc} \leq \Phi_\zeta^* g \leq A g_{\Euc}$,
\item $A^{-1/2}\norm{w-u} \leq \dist_g(\Phi_\zeta(w),\Phi_\zeta(u)) \leq A^{1/2} \norm{w-u}$, 
\item $\Lc(\phi_\zeta) \geq g$ on $\Phi_\zeta(\Bb)$, and 
\item $-A \leq \phi_\zeta \leq 0$. 
\end{enumerate}

Fix $\delta \in (0,1/2)$ and a smooth function $\chi : \Bb \rightarrow [0,1]$ with $\chi(z) = 1$ when $\norm{z} \leq \delta$ and $\chi(z) = 0$ when $\norm{z} \geq 2\delta$. Then pick $C  > 0$ such that 
\begin{align*}
\Lc\left( \chi(w) \log \norm{w}\right) \geq -C g_{\Euc}
\end{align*}
when $\delta \leq \norm{w} \leq 2\delta$. 

Next fix $M>AC$ and define the functions
\begin{align*}
u_\zeta(z) = \chi(\Phi_\zeta^{-1}(z)) \log  \norm{\Phi_\zeta^{-1}(z)} + M \phi_\zeta(z).
\end{align*}
We claim that each $u_\zeta$ is plurisubharmonic. Since $\phi_\zeta$ is plurisubharmonic and the support of the first term is contained in $\Phi_\zeta(\Bb)$, it suffices to consider the functions 
\begin{align*}
v_\zeta(w)= u_\zeta \circ \Phi_\zeta(w) =  \chi(w) \log \norm{w}  + M \phi_\zeta(\Phi_\zeta(w))
\end{align*}
on $\Bb$. By assumption
\begin{align*}
\Lc(\phi_\zeta \circ \Phi_\zeta) =\Phi_\zeta^* \Lc(\phi_\zeta) \geq \Phi_\zeta^*g \geq A^{-1} g_{\Euc}.
\end{align*}
Then since $M > AC$, each $v_\zeta$ is plurisubharmonic. Hence $u_\zeta$ is plurisubharmonic. 

Since $\Phi_\zeta^{-1}$ is well defined and smooth in a neighborhood of $\zeta$, $u_\zeta - \log \norm{z-\zeta}$ is bounded from above in a neighborhood of $\zeta$ and so
\begin{align*}
\Gf_\Omega(z,\zeta) \geq u_\zeta(z) \geq  \log \norm{\Phi_\zeta^{-1}(z)}-MA
\end{align*}
when $z  \in \Phi_\zeta(\delta\Bb)$. Further, 
\begin{align*}
\Gf_\Omega(z,\zeta) \leq \Gf_{\Phi_\zeta(\Bb)}(z,\zeta)=\Gf_{\Bb}(\Phi_\zeta^{-1}(z),0) = \log \norm{\Phi_\zeta^{-1}(z)}
\end{align*}
when $z  \in \Phi_\zeta(\Bb)$.

Finally, since
\begin{align*}
\abs{ \log  \norm{\Phi_\zeta^{-1}(z)} - \log \dist_g(z,\zeta) } \leq \frac{1}{2} \log (A)
\end{align*}
when $z \in \Phi_\zeta(\Bb)$, we have
\begin{align*}
\log \dist_g(z,\zeta) - \left(MA+\frac{1}{2} \log (A)\right) \leq  \Gf_\Omega(z,\zeta) \leq \log \dist_g(z,\zeta) + \frac{1}{2} \log (A)
\end{align*}
when $z \in \Phi_\zeta(\delta\Bb)$. Since $\{ z \in \Omega : \dist_g(z,\zeta) <\tau\} \subset \Phi_\zeta(\delta \Bb)$ when $\tau < A^{-1/2} \delta$ this completes the proof.

\end{proof}

\section{The Kobayashi metric}\label{sec:Kobayashi}

In this section we use the estimates on the pluricomplex Green function in Theorem~\ref{thm:green_fcn_bds} to bound the Kobayashi metric on a domain with bounded intrinsic geometry. 

\begin{definition} Suppose $\Omega \subset \Cb^d$ is a domain. The \emph{(infinitesimal) Kobayashi metric} is the pseudo-Finsler metric
\begin{align*}
k_{\Omega}(z;v) = \inf \left\{ \abs{\xi} : \xi \in \Cb, \ \varphi : \Db \rightarrow \Omega \text{ holo.}, \ \varphi(0) = z, \ \varphi^\prime(0)\xi = v \right\}.
\end{align*}
\end{definition}

We will frequently use the following basic fact.

\begin{observation} If $\Omega_1 \subset \Cb^{d_1}$, $\Omega_1 \subset \Cb^{d_1}$, and $f : \Omega_1 \rightarrow \Omega_2$ is a holomorphic map, then
\begin{align*}
k_{\Omega_2}(f(z); f^\prime(z)v) \leq k_{\Omega_1}(z;v)
\end{align*}
for all $z \in \Omega_1$ and $v \in \Cb^{d_1}$.
\end{observation}

The main result in this section is the following.

\begin{theorem}\label{thm:comp_to_kob} Suppose $\Omega \subset \Cb^d$ is a domain with bounded intrinsic geometry and $g$ is a complete K\"ahler metric on $\Omega$ satisfying Definition~\ref{defn:BBG}. Then there exists $C > 1$ such that 
\begin{align*}
\frac{1}{C} \sqrt{g_{z}(v,v)}\leq  k_\Omega(z;v) \leq C\sqrt{g_{z}(v,v)}
\end{align*}
for all $z \in \Omega$ and $v \in \Cb^d$. In particular, the Kobayashi metric induces a Cauchy complete distance on $\Omega$. 
\end{theorem}

\begin{remark} We will establish the lower bound on the Kobayashi metric using the estimates on the pluricomplex Green function in Theorem~\ref{thm:green_fcn_bds} and the monotonicity of the pluricomplex Green function under holomorphic maps. Alternatively, it is possible to obtain this estimate using the Sibony metric~\cite{Sib1981}. In particular, the Sibony metric is smaller than the Kobayashi metric and one can modify the functions $u_\zeta$ constructed in the proof of Theorem~\ref{thm:green_fcn_bds} to obtain a lower bound on the Sibony metric. 

\end{remark}

Before proving Theorem~\ref{thm:comp_to_kob} we establish a corollary. 

\begin{corollary}\label{cor:pseudoconvex}
A domain with bounded intrinsic geometry is pseudoconvex. 
\end{corollary}

\begin{proof} 
Suppose $\Omega \subset \Cb^d$ is a domain with bounded intrinsic geometry. Let $\dist_{\Omega,K}$ denote the Kobayashi distance on $\Omega$. By Theorem~\ref{thm:comp_to_kob} the metric space $(\Omega, \dist_{\Omega,K})$ is Cauchy complete. Then a result of Royden~\cite[Corollary pg. 136]{R1971} says that $\Omega$ is taut. Then $\Omega$ is pseudoconvex by a result of Wu~\cite[Theorem F]{W1967}.
\end{proof}

\begin{proof}[Proof of Theorem~\ref{thm:comp_to_kob}] Let $\dist_g$ denote the distance induced by $g$. By Theorem~\ref{thm:charts} there exist $A > 1$ and holomorphic embeddings $\Phi_\zeta : \Bb \rightarrow \Omega$ such that 
\begin{enumerate}
\item $\Phi_\zeta(0)=\zeta$,
\item $A^{-1} g_{\Euc} \leq \Phi_\zeta^* g \leq A g_{\Euc}$,
\item $A^{-1/2}\norm{w-u} \leq \dist_g(\Phi_\zeta(w),\Phi_\zeta(u)) \leq A^{1/2} \norm{w-u}$.
\end{enumerate} By Theorem~\ref{thm:green_fcn_bds} there exists $C>0$ and $\tau \in (0,1)$ such that 
\begin{align*}
\Gf_\Omega(\Phi_\zeta(w), \zeta) \geq \log \norm{w} - C
\end{align*}
for all $\zeta \in \Omega$ and $w \in \tau\Bb$. 

Now fix $\zeta \in \Omega$ and $v \in \Cb^d$. Let $w \in \Cb^d$ be the unique vector with $\Phi_\zeta^\prime(0)w = v$. Then 
\begin{align*}
\norm{w} \leq  \sqrt{Ag_\zeta(v,v)}.
\end{align*}
So we can define a holomorphic map $\varphi : \Db \rightarrow \Omega$ by 
\begin{align*}
\varphi(z) = \Phi_\zeta\left(  \frac{w}{\sqrt{Ag_\zeta(v,v)}} z\right). 
\end{align*}
Then $\varphi(0)=\zeta$ and $\varphi^\prime(0) \xi = v$ where $\xi = \sqrt{Ag_\Omega(v,v)}$. So by definition
\begin{align*}
k_\Omega(\zeta;v) \leq \sqrt{Ag_\zeta(v,v)}.
\end{align*} 

For the other direction, fix $m \in \Nb$ and let $\varphi : \Db \rightarrow \Omega$ be a holomorphic map with $\varphi(0)=\zeta$, $v = \varphi^\prime(0)\xi$, and
\begin{align*}
\abs{\xi} \leq \frac{1}{m}+k_\Omega(\zeta;v).
\end{align*}
Then fix $\epsilon > 0$ such that $\varphi(\epsilon \Db) \subset \Phi_\zeta(\tau\Bb)$. Then for $z \in \epsilon \Db$ we have 
\begin{align*}
\log\abs{z}=\Gf_{\Db}(z,0) \geq \Gf_\Omega(\varphi(z), \zeta) \geq \log \norm{\Phi_\zeta^{-1}(\varphi(z))} - C.
\end{align*}
So $\norm{ (\Phi_\zeta^{-1}\circ \varphi)(z)} \leq e^C \abs{z}$ when $z \in \epsilon \Db$. Thus $\norm{ (\Phi_\zeta^{-1}\circ \varphi)^\prime(0)} \leq e^C$. So 
\begin{align*}
 \sqrt{g_{\Omega,\zeta}(v,v)}&  \leq \sqrt{A}\norm{(\Phi_\zeta^{-1})^\prime(\zeta)v}
= \sqrt{A}\norm{ (\Phi_\zeta^{-1}\circ \varphi)^\prime(0)\xi} \leq \sqrt{A}e^C\abs{\xi} \\
 & \leq \sqrt{A}e^C\left( \frac{1}{m}+k_\Omega(\zeta;v) \right). 
\end{align*} 
Since $m$ is arbitrary, then $ \sqrt{g_\zeta(v,v)} \leq \sqrt{A}e^Ck_\Omega(\zeta;v)$.

\end{proof}

\section{Extending holomorphic functions defined on local charts}\label{sec:extensions}

Suppose $\Omega \subset \Cb^d$ is domain with bounded intrinsic geometry and $g$ is a complete K\"ahler metric on $\Omega$ satisfying Definition~\ref{defn:BBG}. Then let $\Phi_\zeta : \Bb \rightarrow \Omega$ be holomorphic embeddings satisfying Theorem~\ref{thm:charts}.

\begin{theorem}\label{thm:extension} For any $m \geq 0$, there exists $C=C(m) > 0$ such that: if $\zeta \in \Omega$, $f : \Phi_\zeta(\Bb) \rightarrow \Cb$ is holomorphic, and $\int_{\Phi_\zeta(\Bb)} \abs{f}^2 dz < \infty$, then there exists a holomorphic function $F : \Omega \rightarrow \Cb$ where
\begin{align*}
\frac{\partial^{\abs{\beta}} F}{\partial z^{\beta}}(\zeta) = \frac{\partial^{\abs{\beta}} f}{\partial z^{\beta}}(\zeta) 
\end{align*}
for all multi-indices $\beta$ with $\abs{\beta} \leq m$ and 
\begin{align*}
\int_\Omega \abs{F}^2 dz \leq C\int_{\Phi_\zeta(\Bb)} \abs{f}^2 dz.
\end{align*}
\end{theorem}

The following argument is based on the proof of~\cite[Proposition 8.9]{GW1979} which itself is based on work of H\"ormander~\cite{H1965}. See also \cite[Section 6]{Catlin1989}, \cite[Theorem 3.4]{M1994} and the discussion in Section~\ref{sec:motivations}. 

\begin{proof} By Theorem~\ref{thm:charts} there exist $A > 1$  such that 
\begin{enumerate}
\item $A^{-1} g_{\Euc} \leq \Phi_\zeta^* g \leq A g_{\Euc}$,
\item $A^{-1/2}\norm{w-u} \leq\dist_g(\Phi_\zeta(w),\Phi_\zeta(u))\leq A^{1/2} \norm{w-u}$.
\end{enumerate}
Since $\Omega$ has Property~\ref{item:SBG}, we can increase $A$ and further assume that there exists a $\Cc^2$ function $\lambda : \Omega \rightarrow \Rb$ such that 
\begin{align*}
\frac{1}{A} g \leq \Lc(\lambda) \leq A g
\end{align*}
 and $\norm{\partial \lambda}_{g} \leq A$. 

Let $\chi : \Bb \rightarrow [0,1]$ be a compactly supported smooth function with $\chi \equiv 1$ on a neighborhood of $0$. 

Fix $m \geq 0$, $\zeta \in \Omega$, and a holomorphic function $f : \Phi_\zeta(\Bb) \rightarrow \Cb$ with $\int_{\Phi_\zeta(\Bb)} \abs{f}^2 dz < \infty$. Let $\alpha = \bar{\partial}(\chi_\zeta f) =  f \bar{\partial}(\chi_\zeta)$ where $\chi_\zeta = \chi \circ \Phi_\zeta^{-1}$.  Notice that $\Omega$  is pseudoconvex by Corollary~\ref{cor:pseudoconvex}. So we can apply Theorem~\ref{thm:existence} to $\alpha$ with weights $\lambda_1 = \lambda$ and $\lambda_2 = 2(d+m)\Gf_\Omega(\cdot, \zeta)$. 

Since $\Lc(\lambda)\geq \frac{1}{A}g$, we have 
\begin{align*}
\norm{ \alpha}_{\Lc(\lambda)}& \leq A\norm{\alpha}_{g} = A\abs{ f} \norm{ \bar{\partial}(\chi \circ \Phi_\zeta^{-1})}_{g} \\
& = A \abs{ f} \norm{\bar{\partial}(\chi)}_{\Phi_\zeta^*g} \leq A^2 \abs{ f} \norm{\bar{\partial}(\chi)}.
\end{align*}
Then, since $\norm{\bar{\partial}(\chi)} \equiv 0$ on a neighborhood of $0$, Theorem~\ref{thm:green_fcn_bds} implies that there exists $C > 0$ (independent of $\zeta$, $f$) such that 
\begin{align*}
\norm{ \alpha}_{\Lc(\lambda)}^2 e^{-\lambda_2} \leq C \abs{ f}^2.
\end{align*}
Then after possibly increasing $C$ (while remaining independent of $\zeta$, $f$),  Theorem~\ref{thm:existence} implies the existence of $u \in L^{2,{\rm loc}}(\Omega)$ such that $\bar{\partial}u = \alpha$ and 
\begin{align*}
\int_\Omega \abs{u}^2 dz \leq \int_\Omega \abs{u}^2 e^{-\lambda_2} dz \leq C \int_\Omega \norm{ \alpha}_{\Lc(\lambda)}^2 e^{-\lambda_2}  dz \leq C^2 \int_{\Phi_\zeta(\Bb)} \abs{f}^2 dz.
\end{align*}

Consider $F = \chi_\zeta f - u$. Then $\bar{\partial} F =0$ and so $F$ is holomorphic. Then, since $\chi_\zeta f$ is a smooth function, $u$ is also smooth. Further, since $\dist_\Omega$ is locally Lipschitz, Theorem~\ref{thm:green_fcn_bds} implies that 
\begin{align*}
e^{\lambda_2} = {\rm O}\left(\norm{z-\zeta}^{2(d+m)}\right)
\end{align*}
(the constant in the big {\rm O} notation depends on $\zeta$). Then, since $\int_\Omega \abs{u}^2 e^{-\lambda_2} dz$ is finite, we must have \begin{align*}
\frac{\partial^{\abs{\beta}} u}{\partial z^{\beta}}(\zeta) =0 
\end{align*}
 for all multi-indices $\beta$ with $\abs{\beta} \leq m$. Then, since $\chi \equiv 1$ on a neighborhood of $0$, 
\begin{align*}
\frac{\partial^{\abs{\beta}} F}{\partial z^{\beta}}(\zeta) = \frac{\partial^{\abs{\beta}} f}{\partial z^{\beta}}(\zeta) 
\end{align*}
for all multi-indices $\beta$ with $\abs{\beta} \leq m$.

Finally, note that
\begin{align*}
\int_\Omega \abs{F}^2 dz \leq \int_{\Phi_\zeta(\Bb)} \abs{f}^2 dz + \int_\Omega \abs{u}^2 dz \leq (1+C^2)  \int_{\Phi_\zeta(\Bb)} \abs{f}^2 dz
\end{align*}
and so the proof is complete.

\end{proof}

\section{Local estimates on the Bergman kernel}\label{sec:local_est}

Suppose $\Omega \subset \Cb^d$ is domain with bounded intrinsic geometry and let $g$ be a complete K\"ahler metric on $\Omega$ satisfying Definition~\ref{defn:BBG}. Then let $\Phi_\zeta : \Bb \rightarrow \Omega$ be holomorphic embeddings satisfying Theorem~\ref{thm:charts}. Using these functions we introduce the following ``local Bergman kernels.'' For $\zeta \in \Omega$, define
\begin{align*}
\beta_\zeta &: \Bb \times \Bb \rightarrow \Cb \\
\beta_\zeta&(w_1,w_2) = \Bf_\Omega(\Phi_\zeta(w_1), \Phi_\zeta(w_2)) \det \Phi_\zeta^\prime(w_1) \overline{  \det \Phi_\zeta^\prime(w_2) }.
\end{align*}
We will prove the following local estimates on these functions. 

\begin{theorem}\label{thm:local_kernel}\
\begin{enumerate}
\item There exists $C_0> 1$ such that 
\begin{align*}
\frac{1}{C_0} \leq \beta_\zeta(w,w) \leq C_0
\end{align*}
for all $\zeta \in \Omega$ and $w \in \Bb$. 
\item  If $\delta \in (0,1)$, then for all multi-indices $a,b$ there exists $C_{a,b}=C_{a,b}(\delta) > 0$ such that 
\begin{align*}
\frac{\partial^{\abs{a}+\abs{b}}\beta_\zeta}{\partial u^{a}\partial \bar{w}^b}(u,w) \leq C_{a,b}
\end{align*}
for all $\zeta \in \Omega$ and $u,w \in \delta\Bb$. 
\end{enumerate}
\end{theorem}

The rest of the section is devoted to the proof of the Theorem. By Theorem~\ref{thm:charts} there exist $A > 1$  such that 
\begin{enumerate}
\item $A^{-1} g_{\Euc} \leq \Phi_\zeta^* g \leq A g_{\Euc}$,
\item $A^{-1/2}\norm{w-u} \leq\dist_\Omega(\Phi_\zeta(w),\Phi_\zeta(u))\leq A^{1/2} \norm{w-u}$.
\end{enumerate}
Since $\Omega$ has Property~\ref{item:SBG}, we can increase $A$ and assume there exists a $\Cc^2$ function $\lambda : \Omega \rightarrow \Rb$ such that 
\begin{align*}
\frac{1}{A} g \leq \Lc(\lambda) \leq A g
\end{align*}
 and $\norm{\partial \lambda}_{g} \leq A$.  

\begin{lemma}\label{lem:unif_local_bd_1}  There exists $c_0> 1$ such that 
\begin{align*}
\Bf_{\Omega}(\zeta, \zeta) \leq \Bf_{\Phi_\zeta(\Bb)}(\zeta, \zeta) \leq c_0\Bf_{\Omega}(\zeta, \zeta)
\end{align*}
for all $\zeta \in \Omega$. 
\end{lemma} 

\begin{proof} We use the following interpretation of the Bergman kernel: if $D \subset \Cb^d$ is a domain and $z \in D$, then 
\begin{align*}
\Bf_D(z,z) = \sup\left\{ \abs{f(z)}^2 : f \in {\rm H}(D), \, \norm{f}_D \leq 1\right\}
\end{align*}
where ${\rm H}(D)$ is the space of holomorphic functions $D \rightarrow \Cb$ and $\norm{\cdot}_D$ is the $L^2$ norm on $D$. Then Theorem~\ref{thm:extension} implies that there exists $c_0 > 1$ such that 
\begin{align*}
\Bf_{\Omega}(\zeta, \zeta) \leq \Bf_{\Phi_\zeta(\Bb)}(\zeta, \zeta) \leq c_0\Bf_{\Omega}(\zeta, \zeta)
\end{align*}
for all $\zeta \in \Omega$. 
\end{proof}

In the next two lemmas we identify $g_{z}$ with the $d$-by-$d$ complex matrix $\left[ g_{z}(\frac{\partial}{\partial z_i}, \frac{\partial}{\partial \bar{z}_j})\right]$. 

\begin{lemma}\label{lem:unif_local_bd_2}  There exists $c_1> 1$ such that 
\begin{align*}
\frac{1}{c_1} \Bf_{\Omega}(\zeta, \zeta) \leq \abs{\det g_{\zeta}}^2 \leq c_1 \Bf_{\Omega}(\zeta, \zeta)
\end{align*}
for all $\zeta \in \Omega$. 
\end{lemma} 

\begin{proof} Notice that
\begin{align*}
\frac{1}{{\rm vol}(\Bb)}&=\Bf_{\Bb}(0,0) = \Bf_{\Phi_\zeta(\Bb)}(\zeta, \zeta) \abs{\det \Phi_\zeta^\prime(0)}^2\\
& =   \Bf_{\Phi_\zeta(\Bb)}(\zeta, \zeta) \abs{ \frac{\det (\Phi_\zeta^*g)_0 }{ \det g_{\zeta}}}^2
\end{align*} 
Then since $A^{-d} \leq \abs{\det (\Phi_\zeta^*g)_0} \leq A^d$, Lemma~\ref{lem:unif_local_bd_1} implies that there exists $c_1 > 1$ such that 
\begin{align*}
\frac{1}{c_1}  \leq \frac{\Bf_{\Omega}(\zeta, \zeta)}{\abs{\det g_{\zeta}}^2} \leq c_1.
\end{align*}
\end{proof}

\begin{lemma}\label{lem:unif_local_bd_3}  There exists $c_2> 1$ such that 
\begin{align*}
\frac{1}{c_2} \leq \beta_\zeta(w,w) \leq c_2
\end{align*}
for all $\zeta \in \Omega$ and $w \in \Bb$. 
\end{lemma}

\begin{proof} Notice that 
\begin{align*}
\beta_\zeta(w,w)&= \Bf_{\Omega}(\Phi_\zeta(w), \Phi_\zeta(w)) \abs{\det \Phi_\zeta^\prime(w)}^2 \\
& =  \Bf_{\Omega}(\Phi_\zeta(w), \Phi_\zeta(w))\abs{ \frac{\det (\Phi_z^*g)_w }{ \det g_{\Phi_\zeta(w)}}}^2.
\end{align*}
So the lemma follows from Lemma~\ref{lem:unif_local_bd_2} and the fact that $A^{-d} \leq \abs{\det (\Phi_z^*g)_w} \leq A^d$.
\end{proof}

\begin{lemma}\label{lem:unif_local_bd_4}  For every $\delta \in (0,1)$ and multi-indices $a,b$ there exists $C=C(\delta, a,b) > 0$ such that 
\begin{align*}
\frac{\partial^{\abs{a}+\abs{b}}\beta_\zeta}{\partial u^{a}\partial \bar{w}^b}(u,w) \leq C
\end{align*}
for all $\zeta \in \Omega$ and $u,w \in \delta\Bb$. 
\end{lemma} 

\begin{proof} Notice that 
\begin{align*}
\abs{\beta_\zeta(u,w)} \leq \sqrt{\beta_\zeta(u,u)}\sqrt{\beta_\zeta(w,w)} \leq c_2. 
\end{align*}
on $\Bb \times \Bb$. Further, $\beta_\zeta$ is holomorphic in the first variable and anti-holomorphic in the second variable. So these estimates follow from Cauchy's integral formulas. 
\end{proof}

\section{The Bergman metric}

In this section we prove Theorem~\ref{thm:bergman_reduction} from the introduction. 

\begin{theorem}\label{thm:bergman_metric_bounded} If $\Omega \subset \Cb^d$ is a domain with bounded intrinsic geometry, then the Bergman metric $g_\Omega$ on $\Omega$ satisfies Definition~\ref{defn:BBG} and
\begin{align*}
\sup_{z \in \Omega} \norm{\nabla^m R}_{g_\Omega} < \infty
\end{align*}
for all $m \geq 0$ where $R$ is the curvature tensor of $g_\Omega$. 
\end{theorem}

The rest of the section is devoted to the proof of Theorem~\ref{thm:bergman_metric_bounded}. Let $\Omega \subset \Cb^d$ be a domain with bounded intrinsic geometry and let $g$ be a complete K\"ahler metric on $\Omega$ which satisfies Definition~\ref{defn:BBG}. 

By Theorem~\ref{thm:charts} there exist $A > 1$ and holomorphic embeddings $\Phi_\zeta : \Bb \rightarrow \Omega$ such that 
\begin{enumerate}
\item $\Phi_\zeta(0)=\zeta$,
\item $A^{-1} g_{\Euc} \leq \Phi_\zeta^* g \leq A g_{\Euc}$,
\item $A^{-1/2}\norm{w-u} \leq \dist_g(\Phi_\zeta(w),\Phi_\zeta(u)) \leq A^{1/2} \norm{w-u}$.
\end{enumerate}

\begin{lemma}\label{lem:comparable} There exists $C > 1$ such that 
\begin{align*}
\frac{1}{C} g \leq g_\Omega \leq C g.
\end{align*}
Hence $g_\Omega$ is complete and satisfies Property~\ref{item:SBG}. 
\end{lemma}

\begin{proof} We use the following interpretation of the Bergman metric: if $D \subset \Cb^d$ is a domain, $z \in D$, and $X \in \Cb^d$, define
\begin{align*}
\eta_D(z;X) = \sup\left\{ \abs{\sum_{j=1}^d x_j \frac{\partial f}{\partial z_j}(z)}^2 : f \in {\rm H}(D), \, \norm{f}_D \leq 1, \, f(z)=0\right\}
\end{align*}
where ${\rm H}(D)$ is the space of holomorphic functions $D \rightarrow \Cb$ and $\norm{\cdot}_D$ is the $L^2$ norm on $D$. 
Then
\begin{align*}
g_{D,z}(X,X)= \frac{1}{\Bf_D(z,z)}\eta_D(z;X).
\end{align*}

By Theorem~\ref{thm:extension} there exists $C > 1$ such that 
\begin{align*}
\eta_{\Omega}(\zeta; X)  \leq \eta_{\Phi_\zeta(\Bb)}(\zeta; X) \leq C \eta_{\Omega}(\zeta; X) 
\end{align*}
for all $\zeta \in \Omega$ and $X \in \Cb^d$. By Lemma~\ref{lem:unif_local_bd_1} and possibly increasing $C > 1$ we can also assume that 
\begin{align*}
\Bf_{\Omega}(\zeta, \zeta) \leq \Bf_{\Phi_\zeta(\Bb)}(\zeta, \zeta) \leq C\Bf_{\Omega}(\zeta, \zeta)
\end{align*}
for all $\zeta \in \Omega$. Thus 
\begin{align*}
\frac{1}{C} g_{\Phi_\zeta(\Bb),\zeta} \leq g_{\Omega,\zeta} \leq C g_{\Phi_\zeta(\Bb),\zeta}
\end{align*}
for all $\zeta \in \Omega$. 

Further,
\begin{align*}
g_{\Phi_\zeta(\Bb), \zeta}=\left(\Phi_\zeta^{-1}\right)^*g_{\Bb,0} =(n+1) \left(\Phi_\zeta^{-1}\right)^*g_{\Euc,0}
\end{align*}
and 
\begin{align*}
\frac{1}{A}\left(\Phi_\zeta^{-1}\right)^*g_{\Euc,0} \leq g_\zeta \leq A\left(\Phi_\zeta^{-1}\right)^* g_{\Euc,0}
\end{align*}
So
\begin{align*}
\frac{n+1}{AC} g_{\zeta} \leq g_{\Omega,\zeta} \leq AC(n+1) g_{\zeta}
\end{align*}
for all $\zeta \in \Omega$. 
\end{proof}

\begin{lemma}\label{lem:curvature_bds} For every $m \geq 0$ 
\begin{align*}
\sup_{z \in \Omega} \norm{\nabla^m R}_{g_\Omega}  < + \infty
\end{align*}
where $R$ is the curvature tensor of $g_\Omega$. 
\end{lemma}

\begin{proof}As in Section~\ref{sec:local_est}, for $\zeta \in \Omega$ define
\begin{align*}
\beta_\zeta &: \Bb \times \Bb \rightarrow \Cb \\
\beta_\zeta&(w_1,w_2) = \Bf_\Omega(\Phi_\zeta(w_1), \Phi_\zeta(w_2)) \det \Phi_\zeta^\prime(w_1) \overline{  \det \Phi_\zeta^\prime(w_2) }.
\end{align*}
Since $w \rightarrow \det \Phi_\zeta^\prime(w)$ is holomorphic, 
\begin{align*}
\Phi_\zeta^* g_\Omega = \sum_{1 \leq i,j \leq d} \frac{\partial^2 \log \beta_\zeta(w,w)}{\partial w_i \partial \bar{w}_j}dw_i d\bar{w}_j.
\end{align*}
So the corollary follows from Theorem~\ref{thm:local_kernel}, Lemma~\ref{lem:comparable}, and  expressing the curvature tensors in local coordinates.
\end{proof}

\begin{lemma} 
$g_\Omega$ has positive injectivity radius. 
\end{lemma}

\begin{proof} By Lemma~\ref{lem:comparable}, 
\begin{align*}
(AC)^{-1} g_{\Euc} \leq \Phi_\zeta^* g_\Omega \leq (AC) g_{\Euc}
\end{align*}
on $\Bb$. Further, $g_\Omega$ has bounded sectional curvature by Lemma~\ref{lem:curvature_bds}. Hence~\cite[Proposition 2.1]{LSY2004b} implies that $g_\Omega$ has positive injectivity radius. 
\end{proof}

\section{The proof of Theorem~\ref{thm:compactness_intro}}

In this section we prove an extension of Theorem~\ref{thm:compactness_intro} from the introduction, but first some general remarks. 

When $\Omega$ is a bounded pseudoconvex domain, a bounded linear operator $S_q : L^2_{(0,q)}(\Omega) \cap \ker \bar{\partial} \rightarrow L^2_{(0,q-1)}(\Omega)$ is said to be a \emph{solution operator for $\bar{\partial}$} if $\bar{\partial}S_q(u) = u$ for all $u \in L^2_{(0,q)}(\Omega) \cap \ker \bar{\partial}$. The operator $\bar{\partial}^* N_q$ is such a solution operator and it is well-known that the compactness of $N_q$ implies the compactness of $\bar{\partial}^* N_q$, see for instance~\cite[Lemma 1]{FS2001}. 

\begin{theorem}\label{thm:compactness}Suppose $\Omega \subset \Cb^d$ is a bounded domain with bounded intrinsic geometry. Then the following are equivalent: 
\begin{enumerate}
\item $\Omega$ satisfies condition $(\wt{P}_q)$. 
\item $N_q$ is compact.
\item There exists a compact solution operator for $\bar{\partial}$ on $(0,q)$-forms.
\item If $g_{\Omega,z}$ is identified with the matrix $\left[ g_{\Omega,z}(\frac{\partial}{\partial z_i}, \frac{\partial}{\partial \bar{z}_j})\right]$, then 
\begin{align*}
\lim_{z\rightarrow \partial\Omega} \sigma_{d-q+1}( g_{\Omega,z})  =\infty.
\end{align*}
\end{enumerate}
If, in addition, $\partial \Omega$ is $\Cc^0$, then the above conditions are equivalent to 
\begin{enumerate}
\setcounter{enumi}{4}
\item $\partial \Omega$ contains no $q$-dimensional analytic varieties.
\end{enumerate}
\end{theorem}

\begin{remark} \ \begin{enumerate}
\item When $\Omega \subset \Cb^d$ is bounded and convex, Fu-Straube~\cite{FS1998} proved that $(1) \Leftrightarrow (2)\Leftrightarrow (3) \Leftrightarrow (5)$.  In the convex case, the equivalence $(4) \Leftrightarrow (5)$ follows from a result of Frankel~\cite{F1991}, see Theorem~\ref{thm:frankel_ii} above.  
\item Recall that 
\begin{align*}
\sigma_1(A) \geq \sigma_2(A) \geq \dots \geq \sigma_d(A)
\end{align*}
denotes the singular values of a $d$-by-$d$ matrix $A$ and so 
\begin{align*}
\sigma_{d-q+1}( g_{\Omega,z}) = \min_{V} \max_{v \in V, v \neq 0} \frac{1}{\norm{v}^2} g_{\Omega,z}(v,v) 
\end{align*}
where the minimum is taken over all $q$-dimensional complex linear subspaces.
\end{enumerate}
\end{remark}

By the remarks proceeding Theorem~\ref{thm:compactness}, $(2) \Rightarrow (3)$ holds for any pseudoconvex domain. McNeal~\cite{M2002} proved that $(1) \Rightarrow (2)$, see Theorem~\ref{thm:McNeal} above. The definition of Property~\ref{item:SBG} almost immediately implies that $(4) \Rightarrow (1)$:

\begin{corollary}[$(4) \Rightarrow (1)$] Suppose $\Omega \subset \Cb^d$ is a bounded pseudoconvex domain whose Bergman metric $g_\Omega$ has Property~\ref{item:SBG}. If 
\begin{align*}
\lim_{z \rightarrow \infty} \sigma_{d-q+1}( g_{\Omega,z}) =\infty, 
\end{align*}
then $\Omega$ satisfies condition $(\wt{P}_q)$.
\end{corollary}

\begin{proof} By hypothesis, there exist $C > 1$ and a $\Cc^2$ function $\lambda : \Omega \rightarrow \Rb$ such that $\frac{1}{C} g_\Omega \leq \Lc(\lambda) \leq C g_\Omega$ and 
\begin{align*}
 \norm{\partial \lambda}_{\Lc(\lambda)} \leq 1.
\end{align*}
Then 
\begin{align*}
\sigma_{d-q+1}(\Lc(\lambda)) \geq\frac{1}{C} \sigma_{d-q+1}(g_{\Omega,z}).
\end{align*}
So $\Omega$ satisfies condition $(\wt{P}_q)$.
\end{proof}

Summarizing our discussion so far, we know that 
\begin{align*}
(4) \Rightarrow (1) \Rightarrow (2)\Rightarrow (3).
\end{align*}
We will complete the proof by showing that $(3) \Rightarrow (4)$ and $(4) \Leftrightarrow (5)$.

For the rest of the section let $\Omega \subset \Cb^d$ be a bounded domain with bounded intrinsic geometry. By Theorems~\ref{thm:charts} and~\ref{thm:bergman_metric_bounded} there exist $A > 1$ and for each $\zeta \in \Omega$ a holomorphic embedding $\Phi_\zeta : \Bb \rightarrow \Omega$ such that 
\begin{align*}
\frac{1}{A} g_{\Euc} \leq \Phi_\zeta^* g_\Omega \leq A g_{\Euc}
\end{align*}
and 
\begin{align}
\label{eq:dist_est_sec_7}
\frac{1}{\sqrt{A}} \norm{w_1-w_2} \leq \dist_\Omega\Big( \Phi_\zeta(w_1), \Phi_\zeta(w_2)\Big) \leq \sqrt{A} \norm{w_1-w_2}
\end{align}
on $\Bb$.

\begin{proposition}[$(5) \Rightarrow (4)$] If $\partial \Omega$ contains no $q$-dimensional analytic varieties, then
\begin{align*}
\lim_{z\rightarrow \partial\Omega }\sigma_{d-q+1}( g_{\Omega,z})=\infty.
\end{align*}
\end{proposition}

\begin{proof} Suppose not. Then there exist $C > 0$, a sequence $(\zeta_m)_{m \geq 1}$ in $\Omega$ converging to $\partial \Omega$, and a sequence $(V_m)_{m \geq 1}$ of $q$-dimensional linear subspaces such that 
\begin{align*}
g_{\Omega,\zeta_m}(v,v) \leq C \norm{v}^2
\end{align*}
for all $v \in V_m$.

By Montel's theorem and passing to a subsequence we can assume that $\Phi_{\zeta_m}$ converges locally uniformly to a holomorphic map $\Phi: \Bb \rightarrow \overline{\Omega}$ with $\Phi(0)\in\partial\Omega$. Since the Bergman metric on $\Omega$ is complete, Equation~\eqref{eq:dist_est_sec_7} implies that $\Phi(\Bb)\subset \partial\Omega$.

For $w \in \Phi_{\zeta_m}^\prime(0)^{-1}V_m$ we have
\begin{align*}
\norm{\Phi_{\zeta_m}^\prime(0)w}^2& \geq \frac{1}{C}  g_{\Omega, \zeta_m}\left( \Phi_{\zeta_m}^\prime(0)w, \Phi_{\zeta_m}^\prime(0)w\right)\\
& =  \frac{1}{C} (\Phi_{\zeta_m}^*g_\Omega)_0(w,w) \geq  \frac{1}{AC}  \norm{w}^2.
\end{align*}
This implies that $\sigma_{q}(\Phi_{\zeta_m}^\prime(0)) \geq (AC)^{-1/2}$. Then 
\begin{align*}
\sigma_{q}(\Phi^\prime(0))=\lim_{m \rightarrow \infty} \sigma_{q}(\Phi_{\zeta_m}^\prime(0)) \geq (AC)^{-1/2}.
\end{align*}
Then, since ${\rm rank}\, \Phi^\prime(0)= \max\{ m : \sigma_m(\Phi^\prime(0)) \neq 0\}$, we see that $\Phi^\prime(0)$ has rank at least $q$. Thus $\partial\Omega$ contains a $q$-dimensional analytic variety and we have a contradiction. 
\end{proof}

\begin{lemma}[$(4) \Rightarrow (5)$] Suppose $\partial \Omega$ is  $\Cc^0$. If 
\begin{align*}
\lim_{z\rightarrow \partial\Omega }\sigma_{d-q+1}( g_{\Omega,z})=\infty,
\end{align*}
then $\partial \Omega$ contains no $q$-dimensional analytic varieties
\end{lemma}

\begin{proof} Suppose not. Then there exists a holomorphic map $\psi : \Bb_q \rightarrow \partial \Omega$ where $\psi^\prime(0)$ has rank $q$ and $\Bb_q \subset \Cb^q$ denotes the unit ball.  By applying a linear change of coordinates to $\Cb^d$, we may assume that $\psi^\prime(0)v=(v,0)$ for all $v \in \Cb^q$. 

Since $\partial \Omega$ is $\Cc^0$ there exists $\nu \in \Cb^d$ and $\epsilon > 0$ such that 
\begin{align*}
t \nu + \psi(\epsilon \Bb_q) \subset \Omega
\end{align*}
for all $t \in (0,\epsilon)$. Let $z_t = t\nu + \psi(0)$. 

We claim that there exists $C > 0$ such that
\begin{align*}
g_{\Omega,z_t}((v,0),(v,0)) \leq C\norm{v}^2
\end{align*}
for all $t \in (0,\epsilon)$ and $(v,0) \in \Cb^q \times \{0\}$. By  Theorem~\ref{thm:comp_to_kob} there exists $C_0>1$ such that 
\begin{align*}
\sqrt{g_{\Omega,z}(v,v)} \leq C_0 k_\Omega(z;v)
\end{align*}
for all $z \in \Omega$ and $v \in \Cb^d$. Define $\psi_t : \Bb_q \rightarrow \Omega$ by $\psi_t(z) = t\nu + \psi(\epsilon z)$. By the definition of the Kobayashi metric, 
\begin{align*}
k_\Omega(z_t; (v,0)) = k_\Omega\left( \psi_t(0); \psi_t^\prime(0) \frac{1}{\epsilon} v\right) \leq k_{\Bb_q}\left(0;   \frac{1}{\epsilon} v\right) = \frac{1}{\epsilon}\norm{v}. 
\end{align*}
Thus 
\begin{align*}
g_{\Omega,z_t}((v,0),(v,0)) \leq \frac{C_0^2}{\epsilon^2}\norm{v}^2
\end{align*}
for all $t \in (0,\epsilon)$ and $(v,0) \in \Cb^q \times \{0\}$. Hence
\begin{align*}
\sigma_{d-q+1}( g_{\Omega,z_t}) \leq \frac{C_0^2}{\epsilon^2}
\end{align*}
for all $t \in (0,\epsilon)$ and we have a contradiction. 
\end{proof}

\begin{proposition}[$(3) \Rightarrow (4)$]\label{prop:3implies4} If  there exists a compact solution operator for $\bar{\partial}$ on $(0,q)$-forms, then 
\begin{align*}
\lim_{z\rightarrow \partial\Omega }\sigma_{d-q+1}( g_{\Omega,z})=\infty.
\end{align*}
\end{proposition}

The rest of the section is devoted to the proof of Proposition~\ref{prop:3implies4}, which is similar to arguments of Catlin~\cite[Section 2]{C1983} and Fu-Straube~\cite[Section 4]{FS1998}. 

Assume $S_q$ is a compact solution operator for $\bar{\partial}$ on $(0,q)$-forms. We argue by contradiction: suppose there exist $C > 0$, a sequence $(\zeta_m)_{m \geq 1}$ in $\Omega$ converging to $\partial \Omega$, and a sequence $(V_m)_{m \geq 1}$ of $q$-dimensional linear subspaces such that 
\begin{align}
\label{eq:gOmega_sing_val_estimate}
g_{\Omega,\zeta_m}(v,v) \leq C \norm{v}^2
\end{align}
for all $v \in V_m$. 

For each $m \geq 0$ let $U_m$ be a unitary matrix with 
\begin{align*}
U_m V_m =  \Cb^q \times \{ 0\}
\end{align*}
and consider the $(0,q)$-forms 
\begin{align*}
\alpha_m = \frac{\Bf_\Omega(\cdot, \zeta_m)}{\sqrt{\Bf_\Omega(\zeta_m, \zeta_m)}} U_m^*(d\bar{z}_1 \wedge \dots \wedge d\bar{z}_q)
\end{align*}
on $\Omega$. Then $\norm{\alpha_m}_2 = 1$ and $\bar{\partial} \alpha_m =0$. So $h_m = S_q(\alpha_m)$ is well defined and by passing to a subsequence we can suppose that $h_m$ converges in $L^2_{(0,q-1)}(\Omega)$. Since $h_m$ converges,  for any $\epsilon > 0$ there exists a compact subset $K \subset \Omega$ such that 
\begin{align}
\label{eq:unif_estimate}
\sup_{m \geq 0} \int_{\Omega \setminus K} \norm{h_m}^2 dz < \epsilon.
\end{align}
We will derive a contradiction by showing that 
\begin{align*}
\int_{B_\Omega(\zeta_m;r)} \norm{h_m}^2 dz
\end{align*}
is uniformly bounded from below. Since $\zeta_m \rightarrow \partial \Omega$ and the Bergman metric is proper, this will contradict Equation~\eqref{eq:unif_estimate}.

Let $\Phi_m := \Phi_{\zeta_m}$. By precomposing each $\Phi_{m}$ with a unitary transformation we may assume that 
\begin{align*}
\Phi_m^\prime(0)\Big( \Cb^q \times \{ 0\}\Big)= V_m.
\end{align*}
Thus $U_m\Phi_m^\prime(0) \Big( \Cb^q \times \{ 0\}\Big)=  \Big( \Cb^q \times \{ 0\}\Big)$. 

Next define $J_m : \Bb \rightarrow \Cb$ by 
\begin{align*}
J_m(w) =\ip{ (U_m\Phi_m)^*(d\bar{z}_1 \wedge \dots \wedge d\bar{z}_q), d\bar{z}_1 \wedge \dots \wedge d\bar{z}_q}.
\end{align*}

\begin{lemma}\label{lem:basic_estimate} \ 
\begin{enumerate}
\item $\overline{J(w)} = \det\left[ \frac{\partial (U_m\Phi_{m})_j}{\partial z_i}(w)\right]_{1 \leq i,j \leq q}$, in particular $J$ is anti-holomorphic. 
\item For any $\delta \in (0,1)$ there exists $C_\delta > 0$ such that 
\begin{align*}
\norm{ \Phi_m^\prime(w)} \leq C_\delta
\end{align*}
and 
\begin{align*}
\abs{J_m(w)} \leq C_\delta
\end{align*}
for all $m \geq 0$ and  $w\in\delta \Bb$.
\item $\inf_{m \geq 0} \abs{J_m(0)} > 0$. 
\end{enumerate}
\end{lemma}

\begin{proof} Notice that 
\begin{align*}
(U_m\Phi_{m})^* d\bar{z}_j = \sum_{i=1}^d \overline{\frac{\partial (U_m\Phi_{m})_j}{\partial z_i}} d\bar{z}_i
\end{align*}
and so part (1) follows from the definition of the determinant.

Part (2) is a consequence of the Cauchy integral formulas and the fact that the functions $ \Phi_m:\Bb \rightarrow \Omega$ are uniformly bounded. 

Since $U_m\Phi_m^\prime(0) \Big( \Cb^q \times \{ 0\}\Big)= \Big( \Cb^q \times \{ 0\}\Big)$, 
\begin{align*}
(U_m \Phi_m)^\prime(0) = U_m \Phi_m^\prime(0) = \begin{pmatrix} L_m & * \\ 0 & * \end{pmatrix}
\end{align*}
where $L_m = \left[\frac{\partial (U_m\Psi_{m})_j}{\partial z_i}(w)\right]_{1 \leq i,j \leq q}$. So if $v \in  \Cb^q \times \{ 0\}$, then Equation~\eqref{eq:gOmega_sing_val_estimate} implies that
\begin{align*}
\norm{L_m v}^2&=\norm{U_m \Phi_m^\prime(0)v}^2=\norm{\Phi_m^\prime(0)v}^2 \geq \frac{1}{C} g_{\Omega,\zeta_m} (\Phi_m^\prime(0)v,\Phi_m^\prime(0)v)\\
& = \frac{1}{C} (\Phi_m^* g_\Omega)_0(v,v) \geq \frac{1}{AC} \norm{v}^2.
\end{align*}
Thus, all the singular values of $L_m$ are greater than $(AC)^{-1/2}$ which implies that 
\begin{equation*}
\abs{J_m(0)}=\abs{ \det(L_m)} = \prod_{j=1}^q \sigma_j(L_m) \geq (AC)^{-q/2}. \qedhere
\end{equation*}
\end{proof}

Define $\wt{\alpha}: \Bb \rightarrow \Cb$ by 
\begin{align*}
\wt{\alpha}_m(w) = \det\left( \Phi_m^\prime(w) \right) \frac{\Bf_\Omega(\Phi_m(w), \zeta_m)}{\sqrt{\Bf_\Omega(\zeta_m, \zeta_m)}} J_m(w)
\end{align*}
and notice that
\begin{align}
\label{eq:alpha_m_ip}
\wt{\alpha}_m =  \det\left( \Phi_m^\prime(w) \right)\ip{ \Phi_m^* \alpha_m, d\bar{z}_1 \wedge \dots \wedge d\bar{z}_q}.
\end{align}

\begin{lemma} After passing to a subsequence, we can assume that $\wt{\alpha}_m$ converges locally uniformly on $\Bb$ to a smooth function $\wt{\alpha}$ and $\wt{\alpha}(0) \neq 0$.
\end{lemma}

\begin{proof} Each $\wt{\alpha}_m$ is a product of a holomorphic function 
\begin{align*}
 f_{m}(w):=\det\left( \Phi_m^\prime(w) \right) \frac{\Bf_\Omega(\Phi_m(w), \zeta_m)}{\sqrt{\Bf_\Omega(\zeta_m, \zeta_m)}}
 \end{align*}
 and an anti-holomorphic function $J_m$. Hence by Montel's theorem it is enough to show that the sequence $\wt{\alpha}_m$ is locally bounded on $\Bb$ and $\abs{\wt{\alpha}_m(0)}$ is uniformly bounded from below. 

Consider, as in Section~\ref{sec:local_est}, the local kernel functions 
\begin{align*}
\beta_{\zeta_m}(w_1,w_2) = \Bf_\Omega(\Phi_m(w_1), \Phi_m(w_2)) \det\left( \Phi_m^\prime(w_1) \right) \overline{\det\left( \Phi_m^\prime(w_2) \right)}.
\end{align*}
From Theorem~\ref{thm:local_kernel} we know that $1 \lesssim \beta_{z_m}(0,0)$ and 
\begin{align*}
\abs{\beta_{z_m}(w,0)} \leq \sqrt{\beta_{z_m}(w,w)\beta_{z_m}(0,0)} \lesssim 1
\end{align*}
for $w \in \Bb$. Then, since
\begin{align*}
 f_{m}(w)=\frac{\beta_{\zeta_m}(w,0)}{\sqrt{\beta_{\zeta_m}(0,0)}} \frac{\overline{\det\left( \Phi_m^\prime(0) \right)}}{\abs{\det\left( \Phi_m^\prime(0) \right)}},
 \end{align*}
the sequence $f_{m}$ is uniformly bounded on $\Bb$ and $\abs{f_{m}(0)}$ is uniformly bounded from below. 

By Lemma~\ref{lem:basic_estimate}, the functions $w \rightarrow J_m(w)$ are uniformly bounded on $\delta \Bb$ for any $\delta < 1$ and $\abs{J_m(0)}$ is uniformly bounded from below. 

\end{proof}

We will finally obtain a contradiction by proving the following. 

\begin{lemma}\label{lem:mass_along_sequence} There exists $r > 0$ such that
\begin{align*}
\liminf_{m\geq 0} \int_{B_\Omega(z_m;r)} \norm{ h_m}^2 dz > 0.
\end{align*}
\end{lemma}

\begin{proof} Since $\wt{\alpha} \neq 0$, there exists a smooth compactly supported function $\psi : \Bb \rightarrow \Cb$ such that 
\begin{align*}
0 < \int_{ \Bb} \wt{\alpha}(w) \overline{\psi(w)} dw.
\end{align*}
Since $\wt{\alpha}_m$ converges uniformly to $\wt{\alpha}$ on the support of $\psi$ we have 
\begin{align*}
0 < \int_{ \Bb} \wt{\alpha}(w) \overline{\psi(w)} dw = \lim_{m \rightarrow \infty} \int_{ \Bb} \wt{\alpha}_m(w) \overline{\psi(w)} dw.
\end{align*}
Then by Equation~\eqref{eq:alpha_m_ip}
\begin{align*}
0 <  \lim_{m \rightarrow \infty} \int_{ \Bb} \ip{  \det\left( \Phi_m^\prime(w) \right)\Phi_m^* \alpha_m, \chi }dw
\end{align*}
where $\chi=\psi d\bar{z}_1 \wedge \dots \wedge d\bar{z}_q$. 

Since 
\begin{align*}
\det\left( \Phi_m^\prime(w) \right)\Phi_m^* \alpha_m & = \det\left( \Phi_m^\prime(w) \right)\Phi_m^*\bar{\partial} h_m =  \det\left( \Phi_m^\prime(w) \right)\bar{\partial} \Phi_m^*h_m \\
& = \bar{\partial} \det\left( \Phi_m^\prime(w) \right)\Phi_m^*h_m,
\end{align*}
we then have 
\begin{align*}
0 < &\lim_{m \rightarrow \infty}\int_{ \Bb} \ip{  \bar{\partial} \det\left( \Phi_m^\prime(w) \right)\Phi_m^*h_m ,\chi} dw  = \lim_{m \rightarrow \infty} \int_{ \Bb} \ip{  \det\left( \Phi_m^\prime(w) \right)\Phi_m^*h_m , \vartheta\chi} dw \\
& \lesssim \liminf_{m \rightarrow \infty} \left(\int_{ \supp(\chi)} \abs{\det\left( \Phi_m^\prime(w) \right)}^2\norm{\Phi_m^* h_m }^2dw\right)^{1/2}
\end{align*}
where $\vartheta$ is the formal adjoint of $\bar{\partial}$. Now 
\begin{align*}
\norm{\Phi_m^* h_m|_w } \leq \norm{\Phi_m^\prime(w)}^{q-1}\norm{h_m|_{\Phi_m(w)} } \lesssim \norm{h_m|_{\Phi_m(w)} }
\end{align*}
for $w \in \supp(\chi)$ by Lemma~\ref{lem:basic_estimate}. So 
\begin{align*}
0 &<  \liminf_{m \rightarrow \infty} \int_{ \Bb} \abs{\det\left( \Phi_m^\prime(w) \right)}^2\norm{h_m|_{\Phi_m(w)} }^2dw
\\
&= \liminf_{m \rightarrow \infty} \int_{\Phi_m( \Bb)} \norm{h_m}^2 dz.
\end{align*}
Finally note that $\Phi_m( \Bb) \subset B_\Omega(\zeta_m;\sqrt{A})$ and so 
\begin{align*}
0 < \liminf_{m \rightarrow \infty} \int_{B_\Omega(\zeta_m;r)} \norm{ h_m}^2 dz
\end{align*}
for any $r \geq \sqrt{A}$.
\end{proof}

\section{Potentials with bounded complex gradients}\label{sec:potentials}

The purpose of this section is to justify the complicated formulation of Property~\ref{item:SBG}. In particular, we consider a stronger, more natural property and then show that is not invariant under biholomorphism. 

\begin{definition} A domain $\Omega \subset \Cb^d$ has Property ($*$) if 
\begin{align*}
\norm{\partial \log \Bf_\Omega(z,z)}_{g_\Omega} 
\end{align*}
is uniformly bounded on $\Omega$.
\end{definition}

Notice that Property ($*$) implies that the Bergman metric has Property~\ref{item:SBG} and is equivalent to: there exists $C > 0$ such that 
\begin{align}
\label{eq:property_b3_prime_equiv}
\abs{\partial\log \Bf_\Omega(z,z)(X)} \leq C\sqrt{g_{\Omega,z}\left( X , X\right)}
\end{align}
for all $X \in \Cb^d$ and $z \in \Omega$.

Property ($*$) seems more natural than Property~\ref{item:SBG}, but unfortunately it is not invariant under biholomorphism. 

\begin{proposition}\label{prop:not_inv} There exists a bounded domain $\Omega$ biholomorphic to $\Db \times \Db$ which does not have Property ($*$) \end{proposition}

\begin{remark} Notice that $\Db \times \Db$ has Property (*) by either direct computation or Proposition~\ref{prop:SBG_convex}.  \end{remark}

The proof requires one lemma.

\begin{lemma} If $F:\Omega_1 \rightarrow \Omega_2$ is a biholomorphism and both $\Omega_1,\Omega_2$ have Property ($*$), then there exists $C > 0$ such that 
\begin{align*}
\abs{\partial\log \abs{\det F^\prime(z)}^2(X) } \leq C\sqrt{g_{\Omega,z}\left( X ,X\right)}
\end{align*}
for all $X \in \Cb^d$ and $z \in \Omega$.
\end{lemma}

\begin{proof} Let $C_1,C_2 > 0$ be constants satisfying Equation~\eqref{eq:property_b3_prime_equiv} for $\Omega_1,\Omega_2$ respectively. Since $\Bf_{\Omega_2}(F(z),F(z)) \abs{\det F^\prime(z)}^2= \Bf_{\Omega_1}(z,z)$ we have
\begin{align*} 
\abs{\partial \log \abs{\det F^\prime(z)}^2(X)}& \leq \abs{\partial\log \Bf_{\Omega_2}(F(z),F(z))(X)} +\abs{\partial\log \Bf_{\Omega_1}(z,z)(X)}  \\
& = \abs{\partial\log \Bf_{\Omega_2}(w,w)( F^\prime(z)X)}_{w=F(z)} +\abs{\partial\log \Bf_{\Omega_1}(z,z)(X)} \\
& \leq C_1 \sqrt{g_{\Omega_2,F(z)}\left( F^\prime(z)X ,F^\prime(z)X\right)}+C_2\sqrt{g_{\Omega_1,z}\left( X ,X\right)} \\
& = (C_1+C_2) \sqrt{g_{\Omega_1,z}\left( X ,X\right)}.
\end{align*}
Notice that in the last equality we used the fact that $F^* g_{\Omega_2} = g_{\Omega_1}$. 
\end{proof}

\begin{proof}[Proof of Proposition~\ref{prop:not_inv}] For a holomorphic function $\psi: \Db \rightarrow \Db-\{0\}$ define
\begin{align*}
F_\psi &: \Db \times \Db \rightarrow \Cb^2 \\
F_\psi &(z_1,z_2) = \left( \psi(z_2)z_1, z_2\right).
\end{align*}
Since $\psi$ is nowhere vanishing, $F$ is injective and hence is a biholomorphism onto its image. Let $\Omega_{\psi} := F_\psi(\Db \times \Db) \subset \Db \times \Db$. We claim that there exists some $\psi$ such that $\Omega_\psi$ does not have Property ($*$). 

Notice that
\begin{align*}
F_\psi^\prime(z) = \begin{pmatrix} \psi(z_2) & \psi^\prime(z_2)z_1 \\  0 & 1 \end{pmatrix}.
\end{align*}
So $\det F_\psi^\prime(z) = \psi(z_2)$ and 
\begin{align*}
\abs{ \frac{\partial}{\partial z_2} \log \abs{\det F_\psi^\prime(z)}^2} = \abs{ \frac{\partial}{\partial z_2} \log \abs{\psi(z_2)}^2}=\abs{\frac{\psi^\prime(z_2)}{\psi(z_2)} }.
\end{align*}
Further, if $g$ is the Bergman metric on $\Db \times \Db$ then 
\begin{align*}
g_{(z_1,z_2)}\left( \frac{\partial}{\partial z_2} , \frac{\partial}{\partial \bar{z}_2} \right) =  \frac{1}{(1-\abs{z_2}^2)^2}.
\end{align*}
So 
\begin{align*}
\frac{\abs{ \frac{\partial}{\partial z_2} \log \abs{\det F_\psi^\prime(z)}^2} }{\sqrt{g_{(z_1,z_2)}\left( \frac{\partial}{\partial z_2} , \frac{\partial}{\partial \bar{z}_2} \right)}}=\abs{\frac{\psi^\prime(z_2)}{\psi(z_2)} }(1-\abs{z_2}^2).
\end{align*}
Thus if we can find $\psi: \Db \rightarrow \Db -\{0\}$ such that the above quantity is unbounded, then $\Omega_\psi$ does not have Property ($*$). 

Let $\psi : \Db \rightarrow \Db -\{0\}$ be a covering map. Then $\psi$ is a infinitesimial isometry relative to the Kobayashi metrics and so 
\begin{align*}
\frac{\abs{\psi^\prime(w)}}{2\abs{\psi(w)} \log \frac{1}{\abs{\psi(w)}}} = \frac{1}{1-\abs{w}^2}
\end{align*}
for all $w \in \Db$. Then 
\begin{align*}
\frac{\abs{\psi^\prime(w)}}{\abs{\psi(w)}}\left(1-\abs{w}^2\right) =  2\log \frac{1}{\abs{\psi(w)}}
\end{align*}
is unbounded since $\psi(\Db) = \Db -\{0\}$. So for this choice of $\psi$, the domain $\Omega_\psi$ does not have Property ($*$). 

\end{proof}

\begin{remark} One can make the above argument more concrete by directly using the explicit covering map $\Db \rightarrow \Db-\{0\}$ given by $\psi(z) = \exp\left( - \frac{1+z}{1-z} \right)$. With this choice, it is possible to explicitly compute the Bergman kernel on $\Omega_\psi$ and then verify directly that $\Omega_\psi$ does not have Property $(*)$.
\end{remark}

\bibliographystyle{alpha}
\bibliography{complex}

\begin{thebibliography}{McN02b}

\bibitem[Bar80]{B1980}
Theodore~J. Barth.
\newblock Convex domains and {K}obayashi hyperbolicity.
\newblock {\em Proc. Amer. Math. Soc.}, 79(4):556--558, 1980.

\bibitem[Ber48]{B1948}
Stefan Bergmann.
\newblock {\em Sur la fonction-noyau d'un domaine et ses applications dans la
  th\'{e}orie des transformations pseudo-conformes}.
\newblock M\'{e}mor. Sci. Math., no. 108. Gauthier-Villars, Paris, 1948.

\bibitem[Bre55]{B1955}
H.~J. Bremermann.
\newblock Holomorphic continuation of the kernel function and the {B}ergman
  metric in several complex variables.
\newblock In {\em Lectures on functions of a complex variable}, pages 349--383.
  The University of Michigal Press, Ann Arbor, 1955.

\bibitem[BS99]{BS1999}
Harold~P. Boas and Emil~J. Straube.
\newblock Global regularity of the {$\overline\partial$}-{N}eumann problem: a
  survey of the {$L^2$}-{S}obolev theory.
\newblock In {\em Several complex variables ({B}erkeley, {CA}, 1995--1996)},
  volume~37 of {\em Math. Sci. Res. Inst. Publ.}, pages 79--111. Cambridge
  Univ. Press, Cambridge, 1999.

\bibitem[Cat83]{C1983}
David Catlin.
\newblock Necessary conditions for subellipticity of the {$\bar \partial
  $}-{N}eumann problem.
\newblock {\em Ann. of Math. (2)}, 117(1):147--171, 1983.

\bibitem[Cat84]{Cat1984b}
David~W. Catlin.
\newblock Global regularity of the {$\bar \partial $}-{N}eumann problem.
\newblock In {\em Complex analysis of several variables ({M}adison, {W}is.,
  1982)}, volume~41 of {\em Proc. Sympos. Pure Math.}, pages 39--49. Amer.
  Math. Soc., Providence, RI, 1984.

\bibitem[Cat89]{Catlin1989}
David~W. Catlin.
\newblock Estimates of invariant metrics on pseudoconvex domains of dimension
  two.
\newblock {\em Math. Z.}, 200(3):429--466, 1989.

\bibitem[Che99]{C1999}
Bo-Yong Chen.
\newblock Completeness of the {B}ergman metric on non-smooth pseudoconvex
  domains.
\newblock {\em Ann. Polon. Math.}, 71(3):241--251, 1999.

\bibitem[Che04]{C2004}
Bo-Yong Chen.
\newblock The {B}ergman metric on {T}eichm\"{u}ller space.
\newblock {\em Internat. J. Math.}, 15(10):1085--1091, 2004.

\bibitem[CS01]{CS2001}
So-Chin Chen and Mei-Chi Shaw.
\newblock {\em Partial differential equations in several complex variables},
  volume~19 of {\em AMS/IP Studies in Advanced Mathematics}.
\newblock American Mathematical Society, Providence, RI; International Press,
  Boston, MA, 2001.

\bibitem[CY80]{CY1980}
Shiu~Yuen Cheng and Shing~Tung Yau.
\newblock On the existence of a complete {K}\"ahler metric on noncompact
  complex manifolds and the regularity of {F}efferman's equation.
\newblock {\em Comm. Pure Appl. Math.}, 33(4):507--544, 1980.

\bibitem[DF83]{DF1983}
Harold Donnelly and Charles Fefferman.
\newblock {$L^{2}$}-cohomology and index theorem for the {B}ergman metric.
\newblock {\em Ann. of Math. (2)}, 118(3):593--618, 1983.

\bibitem[DFW14]{DFW2014}
K.~Diederich, J.~E. Forn{\ae}ss, and E.~F. Wold.
\newblock Exposing points on the boundary of a strictly pseudoconvex or a
  locally convexifiable domain of finite 1-type.
\newblock {\em J. Geom. Anal.}, 24(4):2124--2134, 2014.

\bibitem[DGZ16]{DGZ2016}
Fusheng Deng, Qi'an Guan, and Liyou Zhang.
\newblock Properties of squeezing functions and global transformations of
  bounded domains.
\newblock {\em Trans. Amer. Math. Soc.}, 368(4):2679--2696, 2016.

\bibitem[Don94]{D1994}
Harold Donnelly.
\newblock {$L_2$} cohomology of pseudoconvex domains with complete {K}\"{a}hler
  metric.
\newblock {\em Michigan Math. J.}, 41(3):433--442, 1994.

\bibitem[Don97]{D1997}
Harold Donnelly.
\newblock {$L_2$} cohomology of the {B}ergman metric for weakly pseudoconvex
  domains.
\newblock {\em Illinois J. Math.}, 41(1):151--160, 1997.

\bibitem[FK72]{FK1972}
G.~B. Folland and J.~J. Kohn.
\newblock {\em The {N}eumann problem for the {C}auchy-{R}iemann complex}.
\newblock Princeton University Press, Princeton, N.J.; University of Tokyo
  Press, Tokyo, 1972.
\newblock Annals of Mathematics Studies, No. 75.

\bibitem[Fra91]{F1991}
Sidney Frankel.
\newblock Applications of affine geometry to geometric function theory in
  several complex variables. {I}. {C}onvergent rescalings and intrinsic
  quasi-isometric structure.
\newblock In {\em Several complex variables and complex geometry, {P}art 2
  ({S}anta {C}ruz, {CA}, 1989)}, volume~52 of {\em Proc. Sympos. Pure Math.},
  pages 183--208. Amer. Math. Soc., Providence, RI, 1991.

\bibitem[FS98]{FS1998}
Siqi Fu and Emil~J. Straube.
\newblock Compactness of the {$\overline\partial$}-{N}eumann problem on convex
  domains.
\newblock {\em J. Funct. Anal.}, 159(2):629--641, 1998.

\bibitem[FS01]{FS2001}
Siqi Fu and Emil~J. Straube.
\newblock Compactness in the {$\overline\partial$}-{N}eumann problem.
\newblock In {\em Complex analysis and geometry ({C}olumbus, {OH}, 1999)},
  volume~9 of {\em Ohio State Univ. Math. Res. Inst. Publ.}, pages 141--160. de
  Gruyter, Berlin, 2001.

\bibitem[Gar87]{Gard1987}
Frederick~P. Gardiner.
\newblock {\em Teichm\"{u}ller theory and quadratic differentials}.
\newblock Pure and Applied Mathematics (New York). John Wiley \& Sons, Inc.,
  New York, 1987.
\newblock A Wiley-Interscience Publication.

\bibitem[Gro91]{Gromov1991}
M.~Gromov.
\newblock K\"{a}hler hyperbolicity and {$L_2$}-{H}odge theory.
\newblock {\em J. Differential Geom.}, 33(1):263--292, 1991.

\bibitem[Gro07]{G2007}
Misha Gromov.
\newblock {\em Metric structures for {R}iemannian and non-{R}iemannian spaces}.
\newblock Modern Birkh\"{a}user Classics. Birkh\"{a}user Boston, Inc., Boston,
  MA, english edition, 2007.
\newblock Based on the 1981 French original, With appendices by M. Katz, P.
  Pansu and S. Semmes, Translated from the French by Sean Michael Bates.

\bibitem[GW79]{GW1979}
R.~E. Greene and H.~Wu.
\newblock {\em Function theory on manifolds which possess a pole}, volume 699
  of {\em Lecture Notes in Mathematics}.
\newblock Springer, Berlin, 1979.

\bibitem[H\"65]{H1965}
Lars H\"{o}rmander.
\newblock {$L^{2}$} estimates and existence theorems for the {$\bar \partial $}
  operator.
\newblock {\em Acta Math.}, 113:89--152, 1965.

\bibitem[Her99]{H1999}
Gregor Herbort.
\newblock The {B}ergman metric on hyperconvex domains.
\newblock {\em Math. Z.}, 232(1):183--196, 1999.

\bibitem[HI97]{HI1997}
Gennadi~M. Henkin and Andrei Iordan.
\newblock Compactness of the {N}eumann operator for hyperconvex domains with
  non-smooth {$B$}-regular boundary.
\newblock {\em Math. Ann.}, 307(1):151--168, 1997.

\bibitem[Kli85]{K1985}
M.~Klimek.
\newblock Extremal plurisubharmonic functions and invariant pseudodistances.
\newblock {\em Bull. Soc. Math. France}, 113(2):231--240, 1985.

\bibitem[KO07]{KO2007}
Chifune Kai and Takeo Ohsawa.
\newblock A note on the {B}ergman metric of bounded homogeneous domains.
\newblock {\em Nagoya Math. J.}, 186:157--163, 2007.

\bibitem[Kob59]{K1959}
Shoshichi Kobayashi.
\newblock Geometry of bounded domains.
\newblock {\em Trans. Amer. Math. Soc.}, 92:267--290, 1959.

\bibitem[Kra92]{Krantz1992}
Steven~G. Krantz.
\newblock {\em Partial differential equations and complex analysis}.
\newblock Studies in Advanced Mathematics. CRC Press, Boca Raton, FL, 1992.
\newblock Lecture notes prepared by Estela A. Gavosto and Marco M. Peloso.

\bibitem[KS12]{KS2012}
Kirill Krasnov and Jean-Marc Schlenker.
\newblock The {W}eil-{P}etersson metric and the renormalized volume of
  hyperbolic 3-manifolds.
\newblock In {\em Handbook of {T}eichm\"{u}ller theory. {V}olume {III}},
  volume~17 of {\em IRMA Lect. Math. Theor. Phys.}, pages 779--819. Eur. Math.
  Soc., Z\"{u}rich, 2012.

\bibitem[KZ16]{KZ2016}
Kang-Tae Kim and Liyou Zhang.
\newblock On the uniform squeezing property of bounded convex domains in
  {$\Bbb{C}^n$}.
\newblock {\em Pacific J. Math.}, 282(2):341--358, 2016.

\bibitem[LSY04]{LSY2004a}
Kefeng Liu, Xiaofeng Sun, and Shing-Tung Yau.
\newblock Canonical metrics on the moduli space of {R}iemann surfaces. {I}.
\newblock {\em J. Differential Geom.}, 68(3):571--637, 2004.

\bibitem[LSY05]{LSY2004b}
Kefeng Liu, Xiaofeng Sun, and Shing-Tung Yau.
\newblock Canonical metrics on the moduli space of {R}iemann surfaces. {II}.
\newblock {\em J. Differential Geom.}, 69(1):163--216, 2005.

\bibitem[McM00]{M2000}
Curtis~T. McMullen.
\newblock The moduli space of {R}iemann surfaces is {K}\"{a}hler hyperbolic.
\newblock {\em Ann. of Math. (2)}, 151(1):327--357, 2000.

\bibitem[McN89]{McNeal1989}
Jeffery~D. McNeal.
\newblock Holomorphic sectional curvature of some pseudoconvex domains.
\newblock {\em Proc. Amer. Math. Soc.}, 107(1):113--117, 1989.

\bibitem[McN92]{M1992}
Jeffery~D. McNeal.
\newblock Convex domains of finite type.
\newblock {\em J. Funct. Anal.}, 108(2):361--373, 1992.

\bibitem[McN94]{M1994}
Jeffery~D. McNeal.
\newblock Estimates on the {B}ergman kernels of convex domains.
\newblock {\em Adv. Math.}, 109(1):108--139, 1994.

\bibitem[McN01]{M2001}
Jeffery~D. McNeal.
\newblock Invariant metric estimates for {$\overline\partial$} on some
  pseudoconvex domains.
\newblock {\em Ark. Mat.}, 39(1):121--136, 2001.

\bibitem[McN02a]{M2002b}
Jeffery~D. McNeal.
\newblock {$L^2$} harmonic forms on some complete {K}\"{a}hler manifolds.
\newblock {\em Math. Ann.}, 323(2):319--349, 2002.

\bibitem[McN02b]{M2002}
Jeffery~D. McNeal.
\newblock A sufficient condition for compactness of the
  {$\overline\partial$}-{N}eumann operator.
\newblock {\em J. Funct. Anal.}, 195(1):190--205, 2002.

\bibitem[MV15]{MV2015}
Jeffery~D. McNeal and Dror Varolin.
\newblock {$L^2$} estimates for the {$\overline\partial$} operator.
\newblock {\em Bull. Math. Sci.}, 5(2):179--249, 2015.

\bibitem[NA17]{NA2017}
N.~Nikolov and L.~Andreev.
\newblock Boundary behavior of the squeezing functions of {$\Bbb{C}$}-convex
  domains and plane domains.
\newblock {\em Internat. J. Math.}, 28(5):1750031, 5, 2017.

\bibitem[NPZ11]{NPZ2011}
Nikolai Nikolov, Peter Pflug, and W{\l}odzimierz Zwonek.
\newblock Estimates for invariant metrics on {$\Bbb C$}-convex domains.
\newblock {\em Trans. Amer. Math. Soc.}, 363(12):6245--6256, 2011.

\bibitem[Roy71]{R1971}
H.~L. Royden.
\newblock Remarks on the {K}obayashi metric.
\newblock In {\em Several complex variables, {II} ({P}roc. {I}nternat. {C}onf.,
  {U}niv. {M}aryland, {C}ollege {P}ark, {M}d., 1970)}, pages 125--137. Lecture
  Notes in Math., Vol. 185. Springer, Berlin, 1971.

\bibitem[Shi89]{Shi1989}
Wan-Xiong Shi.
\newblock Deforming the metric on complete {R}iemannian manifolds.
\newblock {\em J. Differential Geom.}, 30(1):223--301, 1989.

\bibitem[Sib81]{Sib1981}
Nessim Sibony.
\newblock A class of hyperbolic manifolds.
\newblock In {\em Recent developments in several complex variables ({P}roc.
  {C}onf., {P}rinceton {U}niv., {P}rinceton, {N}. {J}., 1979)}, volume 100 of
  {\em Ann. of Math. Stud.}, pages 357--372. Princeton Univ. Press, Princeton,
  N.J., 1981.

\bibitem[Sib87]{Sib1987}
Nessim Sibony.
\newblock Une classe de domaines pseudoconvexes.
\newblock {\em Duke Math. J.}, 55(2):299--319, 1987.

\bibitem[Str10]{S2010}
Emil~J. Straube.
\newblock {\em Lectures on the {$\mathscr{L}^2$}-{S}obolev theory of the
  {$\overline{\partial}$}-{N}eumann problem}.
\newblock ESI Lectures in Mathematics and Physics. European Mathematical
  Society (EMS), Z\"{u}rich, 2010.

\bibitem[Wu67]{W1967}
H.~Wu.
\newblock Normal families of holomorphic mappings.
\newblock {\em Acta Math.}, 119:193--233, 1967.

\bibitem[WY20]{WY2020}
Damin Wu and Shing-Tung Yau.
\newblock Invariant metrics on negatively pinched complete {K}\"{a}hler
  manifolds.
\newblock {\em J. Amer. Math. Soc.}, 33(1):103--133, 2020.

\bibitem[Yeu05]{Y2005}
Sai-Kee Yeung.
\newblock Quasi-isometry of metrics on {T}eichm\"{u}ller spaces.
\newblock {\em Int. Math. Res. Not.}, (4):239--255, 2005.

\bibitem[Yeu09]{Y2009}
Sai-Kee Yeung.
\newblock Geometry of domains with the uniform squeezing property.
\newblock {\em Adv. Math.}, 221(2):547--569, 2009.

\end{thebibliography}

\end{document}